\documentclass[12pt,reqno]{amsart}
\usepackage{amssymb,enumerate,mathrsfs, amsmath, amsthm}
\usepackage{url,titletoc,enumitem}
\usepackage{amsmath,amsthm,amscd,amssymb,latexsym,verbatim,cite}

\usepackage[usenames,dvipsnames]{xcolor}
\definecolor{darkblue}{rgb}{0.0, 0.0, 0.55}
\usepackage[pagebackref,colorlinks,linkcolor=BrickRed,citecolor=OliveGreen,urlcolor=darkblue,hypertexnames=true]{hyperref}
\usepackage{cmap}

\renewcommand{\subset}{\subseteq}
\renewcommand{\emptyset}{\varnothing}

\newtheorem{theorem}{Theorem}[section]

\newtheorem{lemma}[theorem]{Lemma}

\newtheorem*{lemma*}{Lemma}

\def\beq{\begin{equation}}
\def\eeq{\end{equation}}
\def\lb{\label}

\numberwithin{equation}{section}

\def\beq{\begin{equation}}
\def\eeq{\end{equation}}

\def\bbR{ {\mathbb R}}

\def\bbD{ {\mathbb D}}
\def\bbC{ {\mathbb C}}

\def\calT{ {\mathcal T}}
\def\calS{ {\mathcal S}}

\def\tht{\theta}

\def\eps{\varepsilon}

\usepackage{algorithm}
\usepackage{algorithmic}

\def\beq{\begin{equation}}
\def\eeq{\end{equation}}

\begin{document}

\title[Positive Vorticity Solutions to Euler Equations]
{Uniqueness of Positive Vorticity Solutions \\ to the 2D Euler Equations on Singular Domains}

\author{Zonglin Han and Andrej Zlato\v{s}}

\address{\noindent Department of Mathematics \\ University of
California San Diego \\ La Jolla, CA 92093 \newline Email: \tt
zlatos@ucsd.edu,
zoh003@ucsd.edu}

%\address{\noindent Department of Mathematics \\ University of
%California San Diego \\ La Jolla, CA 92093 \newline Email: \tt
%zlatos@ucsd.edu}

%\thanks{The author acknowledges partial support by the NSF through the grant DMS-0632442}

\begin{abstract} 
We show that  particle trajectories for positive vorticity solutions to the 2D Euler equations on fairly general bounded simply connected domains cannot reach the  boundary in finite time.  This includes domains with possibly nowhere $C^1$ boundaries and having corners with arbitrary angles, and can fail without the sign hypothesis when the domain has large angle corners.  Hence positive vorticity solutions on such domains are Lagrangian, and we also obtain their uniqueness if the vorticity is initially constant near the boundary.
\end{abstract}

\maketitle

\section{Introduction and Main Results}

%We want to show that under some following assumptions on boundary of domain $\Omega$ and $\omega_0$, the partile cannot reach the boundary in finite time, hence we'll could apply Lacave and Zlatos's lemma for uniqueness of the solution. Assumptions and definitions we need for the boundary of the regulated domain $\Omega$:

In this paper we study the Euler equations 
\begin{align} 
%\label{1.1}
 \partial_t u + (u\cdot \nabla) \,u &= -\nabla p,  \label{1.1}
 % \qquad \text{on } (0,\infty) \times \Omega,
 \\  \nabla \cdot u & = 0  \label{1.2}
\end{align}
on simply connected bounded  open domains $\Omega\subseteq\bbR^2$ with singular boundaries and at times $t>0$, with $u$ the fluid velocity and $p$ its pressure.  These PDE model the motion of two-dimensional ideal fluids and it is standard to assume the {\it no-flow} (or {\it slip}) boundary condition 
 \begin{equation} \label{1.3}
 %\label{eq.Euler3}
u\cdot n = 0
%\text{on } [0,\infty)\times \partial\Omega.
 \end{equation}
on $(0,\infty)\times\partial\Omega$ (pointwise when $\partial\Omega$ is $C^1$), with $n$ being the unit outer normal to $\Omega$.  These PDE can be equivalently reformulated as the active scalar equation
\begin{equation}\label{1.4}
\partial_t  \omega + u\cdot \nabla \omega = 0 
\end{equation}
on $(0,\infty)\times\Omega$, with
\[
 \omega:=\nabla\times u= \partial_{x_1} u_{2} - \partial_{x_2} u_{1}
 \]
being the {\it vorticity} of $u$.  This of course means that the velocity  in  \eqref{1.4} is uniquely determined from the vorticity via $u=\nabla^\perp\Delta^{-1}\omega$, and we can then call $\omega$ a solution rather than $u$. 
 
A natural class of solutions are those with bounded $\omega$ \cite{Yudth}, and we provide the definition of weak solutions from this {\it Yudovich class} at the start of the next section.  We will only consider such solutions on the time interval $(0,\infty)$, when they are called {\it global weak solutions}, because they exist for all initial values $\omega_0\in L^\infty(\Omega)$ on very general domains $\Omega$ \cite{GV-Lac} (nevertheless, our results equally apply to solutions on finite time intervals $(0,T)$).  

 It is well known that  the velocity $u$ is spatially log-Lipschitz on each compact $K\subseteq\Omega$ when $\omega$ is bounded (uniformly in time, see \eqref{1.6} below).  Hence for each $x\in\Omega$ there is a unique solution to the ODE 
 \beq \label{1.5}
\frac d{dt} X_t^x = u(t,X_t^x) \qquad\text{and}  \qquad X_0^x=x
\eeq
on an interval $(0,t_x)$ such that 
\[
t_x:=\sup\{t>0\,|\,X_s^x\in\Omega \text{ for all $s\in(0,t)$}\}
\]
 (so if $X_t^x$ reaches $\partial\Omega$, then $t_x$ is the first such time).  That is, $\{X_t^x\}_{t\in[0,t_x)}$ is the Euler particle trajectory for the particle starting at $x\in\Omega$.  We note that while a priori the ODE only holds for almost all $t\in(0,t_x)$ (with $X_t^x$ being continuous in time), $u$ can be shown to be continuous when $\omega$ is bounded, so that \eqref{1.5} in fact holds for each $t\in[0,t_x)$ (see Subsection \ref{SS2.1} below).
 %.  In fact, this holds for solutions from the Yudovich class on general open bounded domains $\Omega$ 
Since \eqref{1.4} is a transport equation, it is then natural to ask whether general weak solutions are transported by $u$ in the sense that  $\omega(t,X_t^x)=\omega_0(x)$ for a.e.~$t\in(0,\infty)$ and a.e.~$x\in \Omega$ such that $t_x>t$.  This is indeed the case \cite[Lemma 3.1]{HanZla}, but that does not a priori exclude the possibility of vorticity creation and depletion on $\partial\Omega$ unless $t_x=\infty$ for a.e.~$x\in\Omega$ (then $\nabla\cdot u\equiv 0$ shows that $|\Omega\setminus \{X_t^x\,|\,x\in\Omega\text{ and }t_x>t\}|=0$).  If both these properties hold, so that $\omega(t,\cdot)$ is the push-forward of $\omega_0$ via $X_t^x$ for each $t\in(0,\infty)$, we call such $\omega$ a {\it Lagrangian solution}.  It is currently an open question whether non-Lagrangian solutions can exist on (sufficiently singular) two-dimensional domains.

Existence of non-Lagrangian solutions would imply non-uniqueness of weak solutions.  But even if all weak solutions are Lagrangian on some domain, this does not immediately yield their uniqueness. In fact,  while weak solutions are known to be unique on rectangles \cite{BDT} and on domains that are $C^{1,1}$ except at finitely many corners that are all exact acute angle sectors  \cite{DT,LMW}, this remains an open question on more singular domains.  The main issue is that the velocity typically is not  log-Lipschitz near corners with angles greater than $\frac\pi 2$, which removes a crucial ingredient from the proof of uniqueness.  However, one can sometimes obtain a partial result via \cite[Proposition 3.2]{LacZla}, which shows that a Lagrangian solution remains unique as long as it remains constant near the singular portion of $\partial\Omega$.  In particular, if $\omega_0$ is constant near $\partial\Omega$ and each Euler particle trajectory associated with a corresponding solution $\omega$ can be shown to never reach $\partial\Omega$ (in which case ${\rm dist}(\{X_t^x\,|\,x\in K\},\partial\Omega)>0$ for any compact $K\subseteq \Omega$ and any $t>0$), then $\omega$ is the unique Lagrangian solution with initial value $\omega_0$.  And if all weak solutions within some class are proved to be Lagrangian, this will yield uniqueness of $\omega$ within that class.  (We note that uniqueness was also proved on domains smooth except at a single obtuse angle corner $Y$ and for certain special $\omega_0\ge 0$ such that $Y\in {\rm supp} \, \omega(t,\cdot)$ can only hold for $t=0$ \cite{AgrNah}.)

The above approach was successfully used by Lacave for solutions $\omega\ge 0$ on domains that are $C^{1,1}$ except at finitely many corners that are all exact sectors with angles $>\frac\pi 2$ \cite{Lacave-SIAM}, and later without the sign restriction on $\omega$ by Lacave and the second author  on domains that are $C^{1,1}$ except at finitely many corners with angles in $(0,\pi)$ \cite{LacZla}, as well as by both authors on domains with much lower boundary regularity  \cite{HanZla} (including infinitely many corners with angles in $(0,\pi)$).  The latter paper in fact contains a sharp criterion for the geometry of $\partial\Omega$ that guarantees that no weak solution has a particle trajectory that reaches $\partial\Omega$ in finite time. This criterion is slightly stronger than exclusion of corners with angles $>\pi$ (in particular, it is satisfied by all convex domains), and it was demonstrated in \cite{KisZla,LacZla} that particle trajectories for bounded $\omega$ on domains that do have such corners can reach $\partial\Omega$ in finite time.  

The examples in \cite{KisZla,LacZla} all involve sign-changing solutions, so in view of \cite{Lacave-SIAM} it is natural to ask whether signed solutions can exhibit such singular behavior on more irregular domains than those considered in \cite{Lacave-SIAM}.
The main result of the present paper is a negative answer to this question
% an extension of the $\omega\ge 0$ results from \cite{Lacave-SIAM}  to 
on much more general domains, allowing both infinitely many corners without size or shape restrictions and considerably less boundary smoothness in-between them.  In particular, we show that positive (and then obviously also negative)  weak solutions on such domains are Lagrangian, with particle trajectories approaching $\partial\Omega$ no faster than double-exponentially, and that these solutions are also unique when $\omega_0$ is  constant near $\partial\Omega$.

Let $\Omega\subseteq\bbR^2$ be a bounded open Lipschitz  domain with $\partial \Omega$ a  Jordan curve, let $L:=|\partial\Omega|$ be the arc-length of $\partial\Omega$, and let $\sigma:[0,L]\to\bbC$ be a counter-clockwise arc-length parametrization of $\partial \Omega$ (so $\sigma(L)=\sigma(0)$).
For any $\tht\in[0,L)$, the {\it unit forward tangent vector} to $\Omega$ at $\sigma(\tht)$ is the unit vector 
%whose argument (as a complex number) is
\beq \lb{7.4}
\bar\tau(\tht) := \lim_{\phi \rightarrow \tht+} \frac{\sigma(\phi)-\sigma(\tht)}  {|\sigma(\phi)-\sigma(\tht)|} ,
\eeq
provided the limit exists (we also let $\bar\tau(L):=\bar\tau(0)$).  If it does for each $\tht\in[0,L)$, and $\bar\tau$ has one-sided limits everywhere on $[0,L]$, then $\Omega$ is said to be {\it regulated}.  In that case $\bar\tau$ is right-continuous, and  if we identify $\bbR^2$  with $\bbC$ and let $\arg(z)\in (-\pi,\pi]$ for $z\neq 0$, then
\beq\lb{1.5'}
\bar \alpha(\tht):= \arg \left( \frac {\bar\tau(\tht)} {\lim_{\phi \rightarrow \tht-} \bar\tau(\phi)} \right) \in(-\pi,\pi)
\eeq
 for $\tht\in(0,L]$ is such that $\pi-\bar\alpha(\tht)$ is the {\it interior angle} of $\Omega$ at $\sigma(\tht)$.
   Note that $\bar\alpha(0)$ is not defined, and $\bar\alpha(\tht)\in (-\pi,\pi)$ for $\tht\in(0,L]$ because $\Omega$ is Lipschitz.  

So corners of $\Omega$ are precisely the points $\sigma(\tht)$ with $\tht\in(0,L]$ and  $\bar\alpha(\tht)\neq 0$, and regulated domains clearly have countably many of them.  
If also $\sum_{\tht\in(0,L]} |\bar\alpha(\tht)|<\infty$, then 
\beq \lb{7.3}
\bar \beta_c(\tht) := \bar \arg \left( \bar\tau(\tht) \right) - \sum_{\tht'\le \tht}  \bar\alpha(\tht') 
\eeq
%We will assume that $\Omega$ is regulated and has corners at  points $\sigma(\bar\tht_1),\sigma(\bar\tht_2),\dots$ with $\bar\tht_1,  \bar\tht_2, \dots\in(0,L]$,  and we let $\bar \alpha_j:=\bar\alpha(\bar\tht_j)$.
%% (these are precisely those with $\bar\alpha(\tht)\neq 0$), 
%If $\sum_j|\bar\alpha_j|<\infty$, then
%\beq \lb{7.3}
%\bar \beta_c(\tht) := \bar \arg \left( \bar\tau(\tht) \right) - \sum_{\bar\tht_j\le \tht} \bar\alpha_j
%\eeq
is a continuous function on $[0,L]$ provided we let $\bar \arg(\bar\tau(\tht))$ be the argument of $\bar\tau(\tht)$ plus an appropriate $\tht$-dependent integer multiple of $2\pi$.   We will also assume that $\bar \beta_c$ is {\it Dini continuous} on $[0,L]$,  that is, it has a modulus of continuity $m:[0,L]\to[0,\infty)$ with
 $ \int_{0}^{L} \frac{m(r)}{r}  dr < \infty$ (i.e., $|\bar\beta_c(\tht)-\bar\beta_c(\tht')|\le m(|\tht-\tht'|)$ holds for all $\tht,\tht'\in[0,L]$).  We recall that any H\" older modulus of continuity is also a Dini modulus.
 We can now state our main result.
 
 \begin{theorem} \lb{T.1.1}
Assume that a bounded open Lipschitz domain $\Omega\subseteq\bbR^2$ with $\partial\Omega$ a Jordan curve is regulated.  Let $\bar\tau$ be the forward tangent vector to $\Omega$ from \eqref{7.4},  let $\bar\alpha$ be from \eqref{1.5'}, and assume that $\sum_{\tht\in(0,L]} | \bar\alpha(\tht)| <\infty$ and  $\bar\beta_c$ from \eqref{7.3} is Dini continuous.    
Consider any $0\le \omega_0 \in L^\infty(\Omega)$ and let $\omega\ge 0$ 
 from the Yudovich class 
  be any global weak solution to the Euler equations on $\Omega$ with initial value $\omega_0$
%on $(0,T)\times \Omega$ for some $T\in(0,\infty]$ 
(such $\omega$ is known to exist by \cite{GV-Lac}).

(i)
% If $T=\infty$ and
We have $t_x=\infty$ for all $x\in\Omega$ and $\{X_t^x\,|\,x\in\Omega\}=\Omega$ for all $t>0$, and there is a constant $C_\omega<\infty$ such that for any $\eps>0$ and all large enough $t> 0$, 
\beq \lb{7.22}
\sup_{{\rm dist}(x,\partial\Omega)\ge \eps} {\rm dist}(X_t^x,\partial\Omega)\ge  \exp({-e^{C_\omega  t}})
\eeq
(except when $\omega\equiv 0$, but then $X_t^x\equiv x$).  Moreover,   $\omega(t,X_t^x)=\omega_0(x)$ for \hbox{a.e.~$(t,x)\in (0,\infty)\times\Omega$} (i.e., $\omega$ is Lagrangian), and  $u$ is continuous on $[0,\infty) \times\Omega$ and \eqref{1.5} holds pointwise.
%for all $(t,x)\in [0,\infty)\times\Omega$. 

(ii)
If ${\rm supp}\,( \omega_0-a)\cap\partial\Omega = \emptyset$ for some $a\ge 0$, then $\omega$  is the unique non-negative weak solution with initial value $\omega_0$.
\end{theorem}

{\it Remarks.}  
1.  Hence the well-known double-exponential bound on the rate of approach of particle trajectories to the boundaries of smooth domains (going back to \cite{Hold, Wol}) still holds on the domains considered here, even though $u$ can be far from log-Lipschitz near $\partial\Omega$ and even unbounded at corners with angles $>\pi$.  A partial explanation is that $\omega\ge 0$ forces $u$ to ``circulate'' around $\partial\Omega$ counter-clockwise, thus keeping any particle trajectory near any corner for only a short time during each passage through its neighborhood.  However, our domains can even have everywhere singular boundaries (e.g., a dense set of corners), so all of $\partial\Omega$ could be the set of potential trouble spots rather than just a few individual corners.
\smallskip

2. Part (i) of this result suggests a natural open question:  is there any planar domain $\Omega$ and a weak solution $\omega\ge 0$ to the Euler equation on it that has a particle trajectory starting inside $\Omega$ and reaching $\partial\Omega$ in finite time?  Of course, a second one is whether such solutions, if they exist, can fail to be Lagrangian (this is currently open even for unsigned $\omega$).
\smallskip

% It is showed in \cite{KisZla,LacZla} that for any $\gamma\in(\pi,2\pi]$, there is a domain smooth everywhere except at a single corner with interior angle $\gamma$ on which particle trajectories for even smooth solutions can reach the boundary in finite time.  The examples provided there change sign, and (i) shows that this is indeed necessary.
 
Let us briefly discuss our approach and its relation to \cite{Lacave-SIAM, LacZla, HanZla}.  In all four papers, the central ingredient is a non-negative Lyapunov functional on $(0,\infty)\times \overline\Omega$ that vanishes only on $(0,\infty)\times \partial\Omega$ and its change on Euler particle trajectories can be controlled sufficiently well to show that it can never become 0 unless it is 0 initially.  Lacave first chose this functional to be the stream function $\Psi:=-\Delta^{-1}\omega$ of the fluid velocity $u$ \cite{Lacave-SIAM} because its rate of change in the flow direction $u$ is 0.  When $\omega$ does not have a sign, then $\Psi$ can vanish inside $\Omega$, and \cite{LacZla, HanZla} therefore  used instead the time-independent function $1-|\calT(x)|$, with $\calT:\Omega\to\bbD$ a Riemann mapping.  In the present paper we consider  again solutions $\omega\ge 0$, and  so revisit the idea of using the stream function.  However,  in Lemmas \ref{L.3.1}--\ref{L.3.5} we obtain sharper and more general estimates on $\Psi$ and $\partial_t\Psi$ than \cite{Lacave-SIAM}, which allows us to include much more general domains, with arbitrary corners as well as considerably less regular boundaries overall.

In the next section we state these estimates and use them to prove Theorem \ref{T.1.1}, leaving the proofs of the estimates and of a formula for $\partial_t\Psi$ for the last two sections. 

%
%In order to state the assumption that we want to make on the continuous part, we first want to define the local modulus of continuity as:
%\begin{enumerate}
%	\item We call a function $m:[0,2\pi]\to[0,\infty)$ with $m(0)=0$  a {\it modulus} if it is continuous, non-decreasing, and satisfies $m(a+b)\le m(a)+m(b)$ for any $a,b\in[0,2\pi]$ with $a+b\le 2\pi$.
%\end{enumerate}
%
%%%%We now let $\beta^\pm_\calT$ be the positive and negative variation functions of $\bar\beta_\calT$, with $\beta^+_\calT(0)=\beta^-_\calT(0)=0$ (this choice of constants is irrelevant).  That is, they are the distribution functions of the positive and negative parts of the Lebesgue-Stieltjes measure whose distribution function is $\bar\beta_\calT$, and $\bar\beta_\calT=\beta^+_\calT-\beta^-_\calT + \bar\beta_\calT(0)$. 
%%Specifically, we will show that if one defines a decomposition $\bar\beta_\calT=\beta_\calT+\tilde\beta_\calT$ with $\beta_\calT$ non-decreasing, then the answer to this question only depends on the modulus of continuity for $\tilde\beta_\calT$.   
%We assume  that $\beta_{c}$ is Dini-Continuous, i.e., the modulus continuity of $\beta_c $, $m(x)$ satisfying
% \beq\lb{dini}
% \int_{0}^{2\pi} \frac{m(x)}{x}  dx \leq C_{\Omega} < \infty,
% \eeq
% 

{\bf Acknowledgements.}
AZ acknowledges  partial support by  NSF grant DMS-1900943 and by a Simons Fellowship.

\section{Proof of Theorem \ref{T.1.1}}

We complete the proof in three steps.  We always assume that $\Omega$ satisfies the hypotheses from Theorem \ref{T.1.1}, and $(\omega,u)$ is a weak solution to \eqref{1.1}--\eqref{1.3} on $(0,\infty)\times \Omega$, as defined next.

\subsection{Weak solutions and space-time differentiability of the stream function} \lb{SS2.1}

We consider here weak solutions to \eqref{1.1}--\eqref{1.3} from the {\it Yudovich class}
\[
%\mathcal Y:= 
\left\{ (\omega,u)\in L^\infty \left( (0,\infty);L^\infty(\Omega)\times L^2(\Omega) \right) \big|\,\, \text{$\omega= \nabla \times  u$ and \eqref{1.2}--\eqref{1.3} all hold weakly} \right\},
% on $\bbR^+\times\partial\Omega$} \right\}.
%\qquad\text{and}\qquad \omega:= \nabla \times u \in L^\infty ((0,\infty)\times \Omega).
\]
where the weak form of \eqref{1.2}--\eqref{1.3} is 
\[
\int_\Omega u(t,x)\cdot\nabla h(x) \, dx = 0 
\qquad \forall h\in H^1_{\rm loc}(\Omega) \text{ with $\nabla h\in L^2(\Omega)$}
\]
for a.e.~$t\in(0,\infty)$ (see \cite{GV-Lac, GV-Lac2}).  It is well-known that $\nabla\times u\in L^\infty ((0,\infty)\times\Omega)$ implies that $u$ is  bounded and log-Lipschitz on any compact $K\subseteq\Omega$ at a.e.~time $t\in (0,\infty)$ (and uniformly in these times), after possibly redefining it on a measure zero spatial set for each such $t$.  If we also redefine $u$ at the exceptional measure-zero set of times (and also at $t=0$), then for any compact $K\subseteq\Omega$ we will have
\beq \lb{1.6}
\sup_{t\ge0}\sup_{x,y\in K} \left( |u(t,x)|+\frac {|u(t,x)-u(t,y)|}{|x-y| \max \{1,-\ln|x-y|\} } \right)<\infty
\eeq
%With the simplifying convention $[0,\infty]:=[0,\infty)$, 
(this is also shown in the proof of Lemma \ref{L.3.6} below).
Let  now $X_t^{s,x}$ for $(s,x)\in[0,\infty)\times\Omega$ be the unique continuous function satisfying
\beq \label{1.7}
\frac d{dt} X_t^{s,x} = u(t,X_t^{s,x}) \qquad\text{and}  \qquad X_s^{s,x}=x
\eeq
a.e.~on the maximal interval $I^{s,x}\subseteq[0,\infty)$ (containing $s$) such that $X_t^{s,x}\in\Omega$ for all $t\in I^{s,x}\setminus\partial I^{s,x}$.  That is, $I^{s,x}$ is the (backward and forward) life-span of the particle trajectory $X_t^{s,x}$.  Of course, $X^{0,x}_t=X^x_t$ and $I^{0,x}=[0,t_x]$ (or $[0,\infty)$ if $t_x=\infty$) for all $x\in\Omega$.

We say that $(\omega,u)$ from the Yudovich class is a {\it weak solution} to \eqref{1.1}--\eqref{1.3} on $(0,T)\times \Omega$ (for some $T\in(0,\infty]$)  with some initial condition $\omega_0\in L^\infty(\Omega)$, if
\beq\label{1.222}
\int_0^{T} \int_\Omega \omega \left(  \partial_t \varphi + u \cdot \nabla \varphi \right) \,dx dt= -\int_\Omega \omega_0 (x)\varphi(0, x) \,dx \qquad \forall \, \varphi \in C_c^\infty \left([0, T) \times \Omega\right).
\eeq
This is in fact the definition of a weak solution $\omega$ to the transport equation \eqref{1.4}  when $u$ is some given vector field, but it is also equivalent to the relevant weak velocity formulation of the Euler equations on $\Omega$ (see \cite[Remark 1.2]{GV-Lac2}).  When $T=\infty$, we call such solutions {\it global}.  Existence of a global weak solution is guaranteed by \cite{GV-Lac}  for any $\omega_0\in L^\infty(\Omega)$ on very general domains (while uniqueness is still open on most singular domains), and so for the sake of notational simplicity we will always assume  that $T=\infty$.

Lemma 3.1 in \cite{HanZla} now shows that for a.e.~$t\in(0,\infty)$, a weak solution $(\omega,u)$ satisfies  $\omega(t,X_t^{x})=\omega_0(x)$ for a.e.~$x\in\Omega$ such that $t_x<t$.  We can therefore redefine $\omega$ on a set of measure 0 so that $\omega(t,X_t^{x})=\omega_0(x)$ holds for all $x\in\Omega$ and all $t\in(0,t_x)$.  Let now $s_1\in(0,\infty)$ be any Lebesgue point of $\omega$ as a function from $(0,\infty)$ to $L^1(\Omega)$.
Replacing $\varphi$ in \eqref{1.222} by $\varphi\psi_\eps$, where $\psi_\eps\in C^\infty ([0, \infty))$ satisfies $\chi_{[s_1,\infty)}\le\psi_\eps\le \chi_{[s_1-\eps,\infty)}$, and taking $\eps\to 0$ shows that $(\omega,u)$ is also a weak solution to \eqref{1.1}--\eqref{1.3} on $(s_1,\infty)\times \Omega$ with initial condition $\omega(s_1,\cdot)$ (i.e., \eqref{1.222} holds with $(0,\omega_0)$ replaced by $(s_1,\omega(s_1,\cdot)$).  Doing the same with any $\varphi \in C_c^\infty \left((0, \infty) \times \Omega\right)$ and $\chi_{[0,s_1]}\le\psi_\eps\le \chi_{[0,s_1+\eps]}$ shows that $(\omega,u)$ is also a weak solution to \eqref{1.1}--\eqref{1.3} on $(0,s_1)\times \Omega$ with terminal condition $\omega(s_1,\cdot)$ (which becomes an initial condition if we reverse the direction of time and replace $u$ by $-u$).  This and Lemma 3.1 in \cite{HanZla} show that we can redefine $\omega$ on a set of measure 0 so that $\omega(t,X_t^{s_1,x})=\omega(s_1,x)$ holds for all $x\in\Omega$ and all $t\in I^{s_1,x}\setminus\partial I^{s_1,x}$ (clearly the values on the curve $(t,X_t^{s_1,x})$ will not change for any $x$ such that $0\in I^{s_1,x}$).  We can continue this way, with $s_2,s_3,\dots$ consecutively in place of $s_1$, where $\{s_j\}_{j\ge 1}$ is dense in $(0,\infty)$.  This allows us to change $\omega$ on a measure zero set so that for all $s\in[0,\infty)$ (and with $\omega(0,\cdot):=\omega_0$) we will from now have
\beq\label{7.8}
\omega \left(t,X_t^{s,x} \right)=\omega(s,x) \qquad \forall x\in\Omega\text{ and } t\in I^{s,x}\setminus\partial I^{s,x}.
\eeq

It is well known that since $\Omega$ is simply connected, $\omega$ from any weak solution $(\omega,u)$ uniquely defines the velocity $u$ via its  {\it stream function}
\[
\Psi(t,\cdot) := -\Delta^{-1} \omega(t,\cdot) 
\]
for all $t\ge 0$ (the negative sign is chosen so that $\Psi\ge 0$ when $\omega\ge 0$).  Namely, after redefinition of $u$ on a measure zero set we have $u = -\nabla^\perp\Psi$, where  $(v_1,v_2)^\perp:=(-v_2,v_1)$ and $\nabla^\perp:=(-\partial_{x_2},\partial_{x_1})$.  We can now use \eqref{7.8} to show that $\Psi$ is space-time differentiable (we postpone the proof of this to the last section).

\begin{lemma} \label{L.3.6}
We have $\Psi\in C^1([0,\infty)\times K)$ for each compact $K\subseteq\Omega$, and $\nabla\Psi = u^\perp$ and 
\beq\lb{7.13}
\partial_t \Psi(t,x) =  - \frac{1}{2\pi} \int_{\Omega}  \left(\frac{\calT(y)-\calT(x)}{|\calT(y)-\calT(x)|^2}  - \frac{\calT(y)-\calT(x)^*}{|\calT(y)-\calT(x)^*|^2}\right)^T D\calT(y)\, u(t,y)\,  \omega(t,y)  dy
\eeq
for each $(t,x)\in [0,\infty)\times\Omega$, where $\calT:\Omega\to\bbD$ is any Riemann mapping.
\end{lemma}

Note that since this shows that now $u=-\nabla^\perp\Psi$ is continuous on $[0,\infty)\times \Omega$, this version of $u$ still satisfies \eqref{1.6}.  Since $u$ is uniquely determined by $\omega$, from now on we will refer  to $\omega$ as a weak solution to \eqref{1.4} (with $u:=\nabla^\perp \Delta^{-1} \omega$), instead of to $(\omega,u)$ as a weak solution to \eqref{1.1}--\eqref{1.3}.

\subsection{Formulation on the unit disc via  Riemann mapping}

Let next $\calT:\Omega\to\bbD$ be a Riemann mapping as in Lemma \ref{L.3.6},  extended continuously to $\partial\Omega$,  and let $\calS:=\calT^{-1}$.  We will now use $\calT$ to rewrite $u$ and $\partial_t \Psi$ in terms of integrals over $\bbD$.
%Let now $\omega\in L^\infty ([0,\infty)\times\Omega)$ be a weak solution to \eqref{1.1}--\eqref{1.3}, with  stream function $\Psi$ and velocity $u=-\nabla^\perp\Psi$.  
%We can use $\calT$ and the arguments around \eqref{7.8} to write
We have
\beq\lb{2.0}
\Psi(t,x) = -\left[\Delta^{-1} \omega(t,\cdot) \right] (x) =  -\frac{1}{2\pi}  \int_{\Omega}  \ln \frac{|\calT(x)-\calT(y)|}{|\calT(x)-\calT(y)^*||\calT(y)|} \, \omega(t,y)dy,
%\\ = & -\frac{1}{2\pi}  \int_{\bbD}  \ln \frac{|\calT(x)-z|}{|\calT(x)-z^*||z|} \det D\calS(z) \, \omega(t,\calS(z)) dz
\eeq
%be the {\it stream function} for $\omega(t,\cdot)$ (the negative sign is to make $\Psi\ge 0$ when $\omega\ge 0$).  
%% then also $\omega(t,\cdot)$ continuous in $L^1(\Omega)$ on $[0,T)$).
%With $(v_1,v_2)^\perp:=(-v_2,v_1)$ and $\nabla^\perp:=(-\partial_{x_2},\partial_{x_1})$ we then have
and then
\beq \lb{2.2}
u(t,x) =-\nabla^\perp \Psi (t,x) = \frac{1}{2\pi} D \calT(x)^T R(t,\calT(x))
\eeq
for any $(t,x)\in[0,\infty)\times\Omega$,
% (after redefining $u$ on a measure zero set), 
where
\beq\lb{7.19}
R(t,\xi) := \int_{\mathbb{D}} \left(\frac{\xi-z}{|\xi-z|^2} - \frac{\xi-z^*}{|\xi-z^*|^2}  \right)^{\perp}  \det D\calS(z)\, \omega(t,\calS(z)) dz
\eeq
for $(t,\xi)\in[0,\infty)\times\bbD$.  We note that the second equality in \eqref{2.2} holds because $\calT=(\calT^1,\calT^2)$ is analytic, which means that 
\beq\lb{9.1}
D\mathcal{T}= \begin{pmatrix} \partial_{x_1} \calT^1 & \partial_{x_2} \calT^1 \\ \partial_{x_1}  \calT^2 & \partial_{x_2} \calT^2 \end{pmatrix}= \begin{pmatrix} \partial_{x_1}  \calT^1 & \partial_{x_2} \calT^1 \\ - \partial_{x_2}  \calT^1 & \partial_{x_1}  \calT^1 \end{pmatrix}
\eeq
 and so for any $v\in\bbR^2$ we have
\[
\left( \begin{pmatrix} \partial_{x_1} \calT^1 & \partial_{x_1} \calT^2 \\ \partial_{x_2}  \calT^1 & \partial_{x_2} \calT^2 \end{pmatrix} v \right)^\perp = \begin{pmatrix} \partial_{x_2}  \calT^2 & -\partial_{x_2} \calT^1 \\ -\partial_{x_1} \calT^2 & \partial_{x_1} \calT^1 \end{pmatrix} v^\perp,
\]
Lemma \ref{L.3.6} and $u\cdot\nabla \Psi\equiv 0$  now
%and $\omega_t = \nabla\cdot(u\omega)$ 
yield for any $x\in\Omega$ and $t\in[0,t_x)$,
% Euler particle trajectory $X^x_t$ associated with $\omega$ the formula 
%(recall Lemma \ref{L.3.6}) 
%\[
%\frac d{dt} \Psi(t,X^x_t) = \Psi_t(t,X^x_t)=\frac{1}{2\pi}  \int_{\Omega} \nabla_y \left( \ln \frac{|\calT(X^x_t)-\calT(y)|}{|\calT(X^x_t)-\calT(y)^*||\calT(y)|} \right) \cdot u(t,y) \omega(t,y)dy
%\]
%after integration by parts.    By using
%\[
% \frac{|\xi-z|}{|\xi-z^*||z|} =  \frac{|z-\xi|}{|z-\xi^*||\xi|}
%\]
%inside the logarithm, it now follows that
\[
\frac d{dt} \Psi(t,X^x_t) = - \frac{1}{2\pi} \int_{\Omega}  \left(\frac{\calT(y)-\calT(X^x_t)}{|\calT(y)-\calT(X^x_t)|^2}  - \frac{\calT(y)-\calT(X^x_t)^*}{|\calT(y)-\calT(X^x_t)^*|^2}\right)^T D\calT(y)\, u(t,y)\,  \omega(t,y)  dy
\]
(the parenthesis is replaced by $\frac{\calT(y)}{|\calT(y)|^2}$ when $\calT(X^x_t)=0$).
If we substitute \eqref{2.2} here and use 
\beq\lb{7.18}
D\calT(y) D\calT(y)^T = \det D\calT(y)\, I_2
\eeq
 (note that $\det D\calT=(\partial_{x_1} \calT^1)^2+(\partial_{x_2} \calT^1)^2>0$),
%also (with $v\in\bbR^2$ the parenthesis above)
%\[
%D\calT(y)v\cdot D\calT(y)u(t,y)=v^T D\calT(y)^T D\calT(y) u(t,y) = \det D\calT(y) \, v \cdot u(t,y),
%\]
%which holds because $D\mathcal{T}(y)$ is of the form $\begin{pmatrix} a & b \\ -b & a \end{pmatrix}$ due to $\calT$ being analytic.  
after a change of variables we  obtain 
\[
\frac d{dt} \Psi(t,X^x_t) 
%& =  \frac{1}{4\pi^2} \int_{\Omega}  \left(\frac{\calT(y)-\calT(X^x_t)}{|\calT(y)-\calT(X^x_t)|^2}  - \frac{\calT(y)-\calT(X^x_t)^*}{|\calT(y)-\calT(X^x_t)^*|^2}\right) \cdot R(t,\calT(y)) \det D\calT(y) \, \omega(t,y)  dy
 =  - \frac{1}{4\pi^2} \int_{\bbD}  \left(\frac{z-\calT(X^x_t)}{|z-\calT(X^x_t)|^2}  - \frac{z-\calT(X^x_t)^*}{|z-\calT(X^x_t)^*|^2}\right) \cdot R(t,z)  \, \omega(t,\calS(z))  dz.
\]
Finally,  from this and the identity
\beq \lb{2.3}
\left| \frac z{|z|^2} - \frac {w}{|w|^2} \right| = \frac {|z-w|} {|z|\,|w|} 
\eeq
for all $z,w\in\bbC\setminus\{0\}$ we see that (with the fraction below replaced by $\frac1{|z|}$ when $\calT(X^x_t)=0$)
\beq \lb{2.4}
\left| \frac d{dt} \Psi(t,X^x_t)  \right| \le \frac{1}{4\pi^2} \int_{\bbD}  \frac{|\calT(X^x_t)-\calT(X^x_t)^*|}{|z-\calT(X^x_t)|\,|z-\calT(X^x_t)^*|} \, |R(t,z)| \, | \omega(t,\calS(z))|  dz.
\eeq

It will also be convenient to re-parametrize the forward tangent vector $\bar\tau$ to $\Omega$ to
\[
\tau(\tht) :=  \lim_{\phi \rightarrow \tht+} \frac{\calS(e^{i\phi})-\calS(e^{i\tht})}  {|\calS(e^{i\phi})-\calS(e^{i\tht})|} ,
\]
with $\tht\in\bbR$.  Then of course $\tau(\tht)=\bar\tau( \Gamma(e^{i\tht}))$ for all $\tht\in\bbR$, where $\Gamma:= \left(\sigma|_{(0,L]}\right)^{-1}\circ \calS$.  We now let $\{\bar\tht_j\}_{j\ge 1} \subseteq (0,L]$ be the set of all points such that $\Omega$ has a corner at $\sigma(\bar \tht_j)$, and define
\[
\tht_j:=\pi+ \arg\left(- \Gamma^{-1}(\bar\tht_j) \right) \in(0,2\pi] \qquad\text{and}\qquad \alpha_j:=\frac{\bar\alpha(\bar\tht_j)}\pi \in(-1,0)\cup(0,1) 
\]
for $j\ge 1$.  That is, $\Omega$ has corners at $\{\calS(e^{i\tht_j})\}_{j\ge 1} $ with angles $\{\pi-\pi\alpha_j\}_{j\ge 1} $.  Then we define 
\[
\beta_c(\tht):=\bar \beta_c(\Gamma(e^{i\tht}))  \qquad\text{and}\qquad \beta_d(\tht):= \pi\sum_{\tht_j \le \tht} \alpha_j
\]
 for $\tht\in(0,2\pi]$ and extend these two functions to $\bbR$ so that   for all $\tht \in \mathbb{R}$ we have
 \[
 \beta_c(\tht + 2\pi) = \beta_c(\tht) + 2\pi \kappa  \qquad\text{and}\qquad  \beta_d(\tht + 2\pi) = \beta_d(\tht) + 2\pi (1-\kappa),
 \]
  where $\kappa:=\frac{\bar\beta_c(L) - \bar\beta_c(0)}{2\pi}$ (which means that $\sum_{\tht\in(0,L]} \bar\alpha(\tht)=2\pi(1-\kappa)$).  Then of course $\beta_c$ is continuous, $\beta_d$ is piecewise constant, and $\beta:=\beta_c+\beta_d$ is the argument of $\tau$ in the sense that $e^{i\beta(\tht)}=\tau(\tht)$ for all $\tht\in\bbR$ (we also have $\beta(\tht + 2\pi) = \beta(\tht) + 2\pi$).

Lemma 1 in \cite{War} shows that $\Gamma$ and $\Gamma^{-1}$ are both H\" older continuous, which means that $\beta_c$ is Dini continuous because $\bar\beta_c$ is.  Indeed, if $\bar m$ is a modulus of continuity for $\bar\beta_c$ , then  $\beta_c$ has modulus of continuity $m(r):=\bar m(Cr^\gamma)$ for some $C,\gamma>0$,
and a simple change of variables shows that $\int_0^1 \frac{\bar m(Cr^\gamma)}r dr<\infty$ if and only if $\int_0^1 \frac{\bar m(r)}r dr<\infty$.

We next state the following important formula for $\det D\calS$.

\begin{lemma} \label{L.3.1}
%Assume the hypotheses of Theorem \ref{T.1.1}.  Then 
We have
\[
\det D\calS(z) =  \det D\calS (0) \, \Pi_{j\ge1} |z-e^{i\theta_j}|^{-2\alpha_j} 
	%\Pi_{j\ge1} |z-e^{i\theta_j}|^{2\alpha^-_j}
	\exp \left( -\frac{2}{\pi} \int_0^{2\pi}  {\rm Im}\, \frac{z} {e^{i\tht}-z} \left(  \beta_{c}(\tht) -\kappa\tht \right) d\tht \right)
\]
for each $z\in\bbD$ (this holds even without $\beta_c$ being Dini continuous), as well as
\[
\sup_{z\in\bbD} \left|  \int_0^{2\pi}  {\rm Im}\, \frac{z} {e^{i\tht}-z} \left(  \beta_{c}(\tht) -\kappa\tht \right) d\tht \right| <\infty. 
\]
\end{lemma}

%{\it Remark.} The first claim does not require Dini continuity of  $\beta_c$.

\begin{proof}
Since $\calS$ is analytic,  $\det D\calS(z)=|\calS'(z)|^2$, where $\calS'$ is the complex derivative when $\calS$ is considered as a function on $\bbC$.  
%The same is true for its inverse $\calS$, and we also have $\calS'(z)=\calT'(\calS(z))^{-1}$.
%From  {\bf (H)} we have that $\bar\beta_\calT$ has one-sided limits at each point (and these are equal at $\tht\in\bbR$ if and only if $\bar\beta_\calT$ is continuous at $\tht$).  This means that $\bar\beta_\calT$ is a {\it regulated function}, and $\Omega$ is therefore a {\it regulated domain} \cite[Section 3.5]{Pomm}.   
Since $\Omega$ is regulated, Theorem~3.15 in \cite{Pomm}  shows that
\[
\mathcal{S}'(z) = |\mathcal{S}'(0)| \, \exp \left(\frac{i}{2\pi} \int_0^{2\pi} \frac{e^{i\tht}+z} {e^{i\tht}-z} \left(  \beta(\tht) -\tht-\frac\pi 2 \right) d\tht \right)
\]
for all $z\in\bbD$, and from $\int_0^{2\pi} \frac{e^{i\tht}+z} {e^{i\tht}-z}\,d\tht = 2\pi \in\bbR$ and ${\rm Im}\, \frac{e^{i\tht}+z} {e^{i\tht}-z} = 2 {\rm Im}\, \frac{z} {e^{i\tht}-z}$ we get
\beq \lb{7.2}
\det D\calS (z)  = \det D\calS (0) \, \exp \left( -\frac{2}{\pi} \,{\rm Im} \int_0^{2\pi}   \frac{z} {e^{i\tht}-z} \left( \beta (\tht) -\tht \right) d\tht \right)
\eeq
(note that $ \beta(\tht) -\tht$ is $2\pi$-periodic).  We split the integral into two parts, one of which is
	\[
	\int_0^{2\pi} \frac{z} {e^{i\tht}-z} \left( \beta_{d}(\tht) - (1-\kappa)\tht \right) d\tht = i \int_0^{2\pi} \ln(1-ze^{-i\tht})\, d\left({\beta}_{d}(\theta)-(1-\kappa)\tht \right),
	\]
	where we used integration by parts.  Since $\int_0^{2\pi} \ln(1-ze^{-i\tht})  d\tht=\ln 1 =0$, it follows that
\begin{align*}
\exp \left( -\frac{2}{\pi} \,	{\rm Im} \int_0^{2\pi}   \frac{z} {e^{i\tht}-z} \big( \beta_{d}(\tht) - (1-\kappa)\tht \big) d\tht \right) 
&= \exp \left( -\frac{2}{\pi} \,\int_0^{2\pi} \ln|e^{i\tht}-z| \,d \beta_{d}(\tht) \right)
\\ & = \Pi_{j\ge 1} |z-e^{i\theta_j}|^{-2\alpha_j} .
\end{align*}
This and \eqref{7.2} prove the first claim.

   	Let  $\tilde{\beta}(\tht) = \beta_c(\tht) -\kappa \tht$, which is also $2\pi$-periodic. If $\beta_c$ has a Dini modulus of continuity $m$, then $\tilde{\beta}$ has Dini modulus  $\tilde{m}(r) := m(r) + |\kappa| r$.
So for any $z\in\bbD$ and  $\tht_z := \arg(z)$ we obtain using ${\rm Im}\, \frac{z} {e^{i(-\tht+\tht_z)}-z} = -{\rm Im}\, \frac{z} {e^{i(\tht+\tht_z)}-z}$ and $\left|\frac{z} {e^{i(\tht+\tht_z)}-z}\right|\le \frac \pi{2|\tht|}$ for any $\tht\in\bbR$ the estimate
%	
%	We may assume $\tilde{\beta}(\tht_z) = 0$, the reason for that is we can minus any constant (including $\tilde{\beta}(\tht_z) $ for any fixed z) to $\tilde{\beta}$ and add the same constant to $\beta_{d}(\tht)-(1-\kappa)\tht$ which appears in the last proof, from the computation in the last proof we can see that plus any constant will not change the integration result. Hence we have:
\[
	\left| \int_0^{2\pi}  {\rm Im}\, \frac{z} {e^{i\tht}-z} \tilde{\beta}(\tht) d\tht \right| = 	 \left| \int_{-\pi}^{\pi}  {\rm Im}\, \frac{z} {e^{i(\tht+\tht_z)}-z} \left( \tilde{\beta}(\tht+\tht_z) - \tilde{\beta}(\tht_z) \right) d\tht \right|  
	\leq \left| \frac \pi 2 \int_{-\pi}^{\pi} \frac{ \tilde{m}(|\tht|)}{|\theta|}    d\tht\right|.
%	&= \exp\left(2  \int_{-\pi}^{\pi} |\frac{ \tilde{\beta}(\tht+\tht_z) - \tilde{\beta}(\tht_z)}{\theta} |   d\tht\right) \\
%	&\leq  \exp\left(2  \int_{-\pi}^{\pi} |\frac{ \tilde{m}(|\tht|) }{\theta} |   d\tht\right) \\
%	&\leq C_2
\]
Since this is finite, the second claim follows.
%	where
%	\begin{enumerate}
%		\item First equality is by the $2-\pi$ periodicity of the integrand,
%		\item First inequality is by the monotonicity of exponential, as well as $ |{\rm Im}\, \frac{z} {e^{i\tht+\tht_z}-z}| \leq |\frac{z} {e^{i\tht+\tht_z}-z}| \leq \frac{\pi}{|\tht|}$ for all $\tht$
%		\item Second equality is by the argument before the equations.
%		\item Last two inequalities is by definition of modulus of continuity and the argument before the equations ($\tilde{m}$ is also Dini)
%	\end{enumerate}
\end{proof}

In view of \eqref{2.4}, \eqref{7.19}, and this lemma,   of particular concern to us will be corners corresponding to $\alpha_j>0$ (i.e.,  those with angles less than $\pi$; note that  the  velocity $u$ on $\Omega$ in fact vanishes at these, while it may be infinite at the other corners).
%We have $\alpha_j>0$ (resp.~$\alpha_j<0$)  when the corner at $\calS(e^{i\tht_j})$ has angle $<\pi$ (resp.~$>\pi$).
We therefore let $\alpha_j^+:=\max\{\alpha_j,0\}$ and define  $\beta_d^+(\tht) := \pi\sum_{\tht_j \le \tht} \alpha_j^+$ for all $\tht\in(0,2\pi]$.  We then extend $\beta_d^+$ to $\bbR$ so that $\beta_d^+(\tht+2\pi)=\beta_d^+(\tht)+\pi\sum_{j\ge 1} \alpha_j^+$, 
and choose $\delta\in(0,\frac{1}{8} ]$ such that 
\beq \lb{1.11}
\frac {\beta_d^+(\tht+3\delta) - \beta_d^+(\tht-3\delta)}\pi \le \alpha_*:= \frac{1+\max_{j\ge 1} \alpha_j^+ }{2}
%\in \left[\frac 12,1 \right)
\eeq
for each $\tht\in\bbR$.  Note that  $\alpha_*\in[\frac 12,1)$ because $\max_{j\ge 1} \alpha_j^+ <1$ by $\sum_{j\ge 1} |\alpha_j| < \infty$.

%  ---------------------
%  
%  Lastly, for any fixed $\xi \in \bbD$, by \eqref{7.9}, we have 
%  $$|\left(\frac{z-\xi}{|z-\xi|^2}  - \frac{z-\xi^*}{|z-\xi^*|^2}\right) \cdot R(t,z)  \, \omega(t,\calS(z))| \leq C\frac{1}{|\xi-z|} (1-|\xi|)^{1-2\alpha_*} $$ 
%  on $B(0,\frac{1+|\xi|}{2})$, and 
%  $$|\left(\frac{z-\xi}{|z-\xi|^2}  - \frac{z-\xi^*}{|z-\xi^*|^2}\right) \cdot R(t,z)  \, \omega(t,\calS(z))| \leq C\frac{1}{1-|\xi|} (1-|z|)^{1-2\alpha_*} $$ 
%  on $\bbD \backslash B(0,\frac{1+|\xi|}{2})$, which are also integrable on each domain respectively, hence we could pass the $\partial_t$ inside the integration. 
%  
%  LACAVE p.18
%  
%------------------

\subsection{Estimates on the stream function and conclusion of the proof}

We now state the following three crucial estimates, whose proofs we postpone to the next section.  In them, constants $C_\Omega$ and $C_\Omega'$  only depend on $\Omega$.

 \begin{lemma} \label{L.3.3}
There is $C_\Omega > 0$ such that  for each $(t,\xi)\in [0,\infty)\times \bbD$ we have
\[
|\Psi(t, \calS(\xi))| \leq C_\Omega \|\omega(t,\cdot)\|_{L^\infty} (1-| \xi |)^{2\min\{1-\alpha_*, 1/4\}}.
\]
\end{lemma}

\begin{lemma} \label{L.3.4}
	If $\omega \ge 0$, then for each $(t,\xi)\in [0,\infty)\times \bbD$ we have 
\[
\Psi(t,\calS(\xi)) \geq \frac{ 1-|\xi|}{100\pi}   \int_{\mathbb{D}} \frac{(1-|z|)  \det D\calS(z)  }{\max\{|z-\xi|, 1-|\xi|\}^2} \, \omega(t,\calS(z)) dz .
\]
\end{lemma}

\begin{lemma}\label{L.3.5}
There is $C_\Omega' > 0$ such that	
%if $\omega \ge 0$, then 
for each $(t,\xi)\in [0,\infty)\times (\bbD\setminus B(0,\frac 34))$
% with $|\xi| \geq 1-\frac{\delta}{8}$ (NEED THIS?) 
we have  
	\begin{align*}
	\int_{\mathbb{D}} \frac{|R(t,z)|}{|z-\xi||z-\xi^*|} dz \leq C_\Omega' |\ln(1-|\xi|)| \left( \int_{\mathbb{D}} \frac{(1-|z|)  \det D\calS(z)  }{\max\{|z-\xi|, 1-|\xi|\}^2} \, |\omega(t,\calS(z))|  dz + \|\omega(t,\cdot)\|_{L^\infty} \right).
	\end{align*}
\end{lemma}

{\it Remarks.}  
1.  Lemmas \ref{L.3.3} and \ref{L.3.4} are sharper and more general versions of Lemmas 3.1 and 3.2 in \cite{Lacave-SIAM}.  Our use of Lemma \ref{L.3.5} to estimate $\partial_t\Psi$ is analogous to the use of Proposition 2.4 and Lemma 3.5 in \cite{Lacave-SIAM}, but instead of bounding $|R|$ above by essentially $\|\omega\|_{L^\infty}$ and leaving $\omega$ as a function, we bound $\omega$ by $\|\omega\|_{L^\infty}$ and leave $|R|$ in \eqref{2.4}.  This is because  for the domains $\Omega$ considered here,  $R$ can blow up at $\partial\bbD$ (see \eqref{7.9} below).   In particular, this happens at corners with angles $\le\frac\pi 2$, which is why such corners had to be excluded in \cite{Lacave-SIAM}.
\smallskip

2. Lemma \ref{L.3.5} easily extends to $\xi \in B(0,\frac 34)$ but we will not need this.
\smallskip

From now assume also that $\omega\ge 0$.
Since $(1-|z|)  \det D\calS(z) $ is bounded below by a positive constant on $B(0,r)$ for any $r<1$ due to Lemma \ref{L.3.1}, for any $a>0$ there is $c_a>0$ such that the second integral in Lemma \ref{L.3.5} is bounded below by $c_a\|\omega(t,\cdot)\|_{L^\infty}$  whenever 
%$\omega(t,\cdot)\ge 0$ and 
\[
 \|\omega(t,\cdot)\|_{L^1} \ge a\|\omega(t,\cdot)\|_{L^\infty}.
\]
From this, the above lemmas, and \eqref{2.4} it follows that when $|\calT(X^x_t)|\ge \frac 34$ (in which case  also $|\calT(X^x_t)-\calT(X^x_t)^*|\le 3 (1-|\calT(X^x_t)|)$), then we have
\begin{align} 
\left| \frac d{dt} \Psi(t,X^x_t)  \right| & \le \frac {75 C_\Omega'}{\pi} \frac{1+c_a}{c_a} \|\omega(t,\cdot)\|_{L^\infty} |\ln(1-|\calT(X^x_t)|)| \Psi(t,X^x_t) \notag
\\ & \le C_{a,\Omega} \|\omega(t,\cdot)\|_{L^\infty}  \Psi(t,X^x_t)  \left| \ln \frac{\Psi(t,X^x_t) }{ C_\Omega \|\omega(t,\cdot)\|_{L^\infty}} \right|, \lb{7.1}
\end{align}
where $C_{a,\Omega}>0$ is some constant that only depends on $(a,\Omega)$. 
%We want to show that the partial cannot reach boundary infinite time by showing that there exists a constant $C$, which only depends on $\Omega$ and $\omega$ such that if $|\calT(x)| > 1-\frac{\delta}{8}$, we have 
%\beq\lb{PsiLogIneq}
% |\partial_t \Psi(t,x)| \leq C  \Psi(t,x)\ln|\Psi(t,x)|.
%\eeq
%
%
%Notice that lemma \ref{L.3.4} implies that $\Psi(t,x) \geq C_{\omega}(1-|\xi|)$ where the constant is time-independent (and x-independent, by choice of $\delta$). Hence combining the computation above and lemma \ref{L.3.4}, \ref{L.3.5}, we have for $|\calT(x)| > 1-\frac{\delta}{8}$
%\begin{align*}
% |\partial_t \Psi(t,x)| &\leq C|\calT(x)-\calT(x)^*|||\omega_0||_{L^\infty} \int_{\Omega}  \frac{|R(t,\calT(y))|}{|\calT(y)-\calT(x)||\calT(y)-\calT(x)^*|} |\det (D\calT(y))| dy \\
%\leq  C_{31} & |\ln(1-|\xi|)| (1-|\xi|) \left( \int_{\mathbb{D}}  \Lambda(z) \frac{1-|z|}{M(\xi,z)^2} dz+1\right) \\
% \leq C_{\omega} & \Psi(t,x)\ln|\Psi(t,x)|\\
% \end{align*}
% where the last the inequality is also because of \ref{constantlb}.

For each $\eps>0$ let $\Omega_\eps:=\Omega\setminus \bigcup_{x\in\partial\Omega}B(x,\eps)$.  For each $\eps>0$ such that $\Omega_{2\eps}\neq\emptyset$,  let 
%$\delta_\eps:={\rm dist}(\Omega_{2\eps},\Omega\setminus \Omega_{\eps})>0$  and  
\[
T_\eps:={\rm dist}(\Omega_{2\eps},\Omega\setminus \Omega_{\eps}) \|u\|_{L^\infty((0,\infty)\times\Omega_\eps)}^{-1} >0.
\]
Then $X^x_t\in\Omega_\eps$ for all $(t,x)\in [0,T_\eps]\times\Omega_{2\eps}$, and therefore \eqref{7.8} yields $\omega(t,X^x_t)=\omega_0(x)$ for all $(t,x)\in [0,T_\eps]\times \Omega_{2\eps}$.
% by Lemma 3.1 in \cite{HanZla} (the lemma only concludes this for a.e.~$t$ but our adjustment to $\omega$ then guarantees that this holds for all $t\in(0,T_\eps)$).  
Taking $\eps\to 0$ we obtain 
\[
\|\omega_0\|_{L^\infty}\le \liminf_{t\to 0} \|\omega(t,\cdot)\|_{L^\infty}\le \|\omega\|_{L^\infty},
\]
%and
%\beq\lb{7.7}
%\|\omega_0\|_{L^1}= {\rm ess}\lim_{t\to 0} \|\omega(t,\cdot)\|_{L^1},
%\eeq
and then from $\nabla \cdot u\equiv 0$ also
\beq\lb{7.5}
\|\omega(t,\cdot)\|_{L^1} \ge \|\omega_0\|_{L^1(\Omega_{2\eps})}
\ge \|\omega_0\|_{L^1} - |\Omega\setminus \Omega_{2\eps}|\, \|\omega_0\|_{L^\infty} 
\ge \|\omega_0\|_{L^1} - |\Omega\setminus \Omega_{2\eps}|\, \|\omega\|_{L^\infty}
\eeq
for each $\eps>0$ and all $t\in[0,T_\eps]$. 

If now $\omega_0\not\equiv 0$, let $a:= \frac 12 \|\omega_0\|_{L^1} \|\omega\|_{L^\infty}^{-1}>0$ and let $\eps>0$ be such that  $|\Omega\setminus \Omega_{2\eps}|\le a$. From \eqref{7.5} we obtain
\[
\|\omega(t,\cdot)\|_{L^1} \ge a \|\omega\|_{L^\infty}\ge a \|\omega(t,\cdot)\|_{L^\infty}
\]
for all $t\in[0,T_\eps]$.  Thus \eqref{7.1} yields
\beq\lb{7.6}
\left| \frac d{dt} \Psi(t,X^x_t)  \right|  \le C_{a,\Omega} \|\omega\|_{L^\infty}  \Psi(t,X^x_t)  \left| \ln \frac{\Psi(t,X^x_t) }{ C_\Omega \|\omega\|_{L^\infty}} \right|
\eeq
for all $(t,x)\in[0,T_\eps]\times \Omega$ such that $|\calT(X^x_t)|\ge \frac 34$.  This and Gronwall's inequality show that $X^x_t\in\Omega$ for all $(t,x)\in[0,T_\eps]\times\Omega$.  Therefore $\omega(t,X^x_t)=\omega_0(x)$ for all $(t,x)\in [0,T_\eps]\times \Omega$, and in particular, $\|\omega(T_\eps,\cdot)\|_{L^1} = \|\omega_0\|_{L^1}$.
% for some $T'\in[\frac{T_\eps}2,T_\eps)$.  
We can therefore repeat this argument with the same $a$ and $\eps$ on the time interval $[T_\eps, 2 T_\eps]$, then on $[2T_\eps, 3T_\eps]$, etc.
%$[T'',T''+T_\eps)$ for some $T''\in [T'+\frac{T_\eps}2,T'+T_\eps)$, etc.  

It follows that $\omega$ is a Lagrangian solution to \eqref{1.4} on  $(0,\infty)\times\Omega$ and $\|\omega(t,\cdot)\|_{L^p}=\|\omega_0\|_{L^p}$ for all $(t,p)\in[0,\infty)\times[1,\infty]$.  Integrating \eqref{7.6} shows that there is a constant $C_\omega$ (depending on $\|\omega_0\|_{L^\infty},\|\omega_0\|_{L^1},\Omega$) such that for each $\eps>0$ and all large enough $t>0$ we have $\Psi(t,X^x_t)\ge \exp(-e^{C_\omega t})$ whenever $\Psi(0,x)\ge\eps$.
%(UNIF double exp decay of $\Psi(t,X^x_t)$).  Conservation of $\|\omega(t,\cdot)\|_{L^1}$ and $\|\omega(t,\cdot)\|_{L^\infty}$, together with 
Since Lemma \ref{L.3.4} yields $C_\omega'>0$ such that $\Psi(t,\calS(\xi))\ge C_\omega'(1-|\xi|)$ for all $(t,\xi)\in[0,\infty)\times\bbD$, and $\calT$ is H\" older continuous on $\overline\Omega$ (see \cite[Lemma 1]{War}), this shows \eqref{7.22}.  Using also that \eqref{1.5} can clearly be solved backwards in time with the same estimate on the boundary approach rate, we find that $\{X_t^x\,|\,x\in\Omega\}=\Omega$, thus finishing the proof of Theorem \ref{T.1.1}(i) for $\omega_0\not\equiv 0$.

If $\omega_0\equiv 0$, then $\omega\equiv 0$ is clearly a Lagrangian solution to \eqref{1.4} on  $(0,\infty)\times\Omega$ with $X^x_t= x$, which satisfies Theorem \ref{T.1.1}(i) except for \eqref{7.22}.  
%If $\omega\not\equiv 0$ is a bounded weak solution on time interval $(0,T)$ for some $T>0$, then the above arguments with time 0 replaced by  any time $T'>0$ such that $\omega(T',\cdot)\not\equiv 0$ show that for all $t\in(T',T)$ we have $\|\omega(t,\cdot)\|_{L^1}=\|\omega(T',\cdot)\|_{L^1}$.  But then $\|\omega(t,\cdot)\|_{L^1}$ must be constant on the time interval $(T'',T)$, where $T''\in[0,T)$ is the infimum of times $t$ with $\omega(t,\cdot)\not\equiv 0$ (so this constant is positive), which contradicts continuity of $\omega(t,\cdot)$ in $L^1(\Omega)$ because $\omega_0\equiv 0$.
%%\eqref{7.7} with time 0 replaced by ${\rm ess}\inf\{ T'\in(0,T)\,|\, \omega(T',\cdot)\not\equiv 0 \}$.  
%Thus $\omega\equiv 0$ is the only weak solution with $\omega_0\equiv 0$.  The first claim in Theorem \ref{T.1.1} is therefore proved.
If $\omega\ge 0$ is a different global weak solution, then the above arguments with time 0 replaced by  any $T'>0$ such that $\omega(T',\cdot)\not\equiv 0$ show that for all $t\in(T',\infty)$ we have $\|\omega(t,\cdot)\|_{L^1}=\|\omega(T',\cdot)\|_{L^1}$.  But then $\|\omega(t,\cdot)\|_{L^1}$ must be constant on the time interval $(T'',\infty)$, where $T''\in[0,\infty)$ is the infimum of times $t$ with $\omega(t,\cdot)\not\equiv 0$ (and that constant is then positive).  This contradicts continuity of $\omega$ as an $L^1(\Omega)$-valued function of time because $\omega(0,\cdot)=\omega_0\equiv 0$.
%%\eqref{7.7} with time 0 replaced by ${\rm ess}\inf\{ T'\in(0,T)\,|\, \omega(T',\cdot)\not\equiv 0 \}$.  
%Thus $\omega\equiv 0$ is the only weak solution with $\omega_0\equiv 0$.  The first claim in Theorem \ref{T.1.1} is therefore proved.

Theorem~\ref{T.1.1}(ii) follows immediately from Theorem \ref{T.1.1}(i) and Proposition 3.2 in \cite{LacZla}.  We note that the latter result shows that Lagrangian solutions are unique as long as they remain constant near $\partial\Omega$  (more specifically, near the non-$C^{2,\gamma}$ portion of $\partial\Omega$ for some $\gamma>0$).

\section{Proofs of Lemmas \ref{L.3.3}--\ref{L.3.5}}

Let us first state an auxiliary technical result.

 \begin{lemma}\label{L.3.0}
	Let $\beta$ be a (positive) measure on $\bbR$ and let $I:= (\tht^*-2\delta, \tht^*+2\delta)$ for some $\tht^*\in\bbR$ and $\delta\in(0,\frac \pi 2]$.
%, and let $\delta\in(0,\frac 16)$ (???).
Let $H \subset 
%B(e^{i\tht^*},\frac{\delta}{2})\cap 
\bbD$ be an open region such that if $re^{i(\tht^*+\phi)} \in H$ for some  $r\in(0,1)$ and $|\phi|\le\pi$, then $re^{i(\tht^*+\phi')} \in H$ whenever $|\phi'|\le|\phi|$ (i.e., $H$ is symmetric and angularly convex with respect to  the line connecting 0 and $e^{i\tht^*}$).
 If $F:[0,\infty)\to[0,\infty)$ is non-decreasing and convex, then 
\[
	\int_{H} f(z) F\left( g(z) + \frac{1}{\beta(I)} \int_{I} h(|z-e^{i\tht}|) d\beta(\tht) \right) dz \leq  \int_{H} f(z) F\left( g(z) + h(|z-e^{i\tht^*}|) \right) dz
\]
	holds for any non-increasing $h:(0,\infty)\to[0,\infty)$ and non-negative $f,g\in L^1(H)$ such that $f(re^{i(\tht^*+\phi')}) \geq f(re^{i(\tht^*+\phi)})$ and $g(re^{i(\tht^*+\phi')}) \geq g(re^{i(\tht^*+\phi)})$ whenever $r\in(0,1)$ and  $|\phi'|\le |\phi|$.
\end{lemma}

The proof of this result is identical to that of Lemma 4.1 in \cite{HanZla}, which was stated with $F(s)=s^\alpha$ for some $\alpha\ge 1$, because the only properties of $F$ used in it were that it is non-decreasing and convex.  We will be using it here with $F(s):=e^s$,  $g\equiv 0$, and  $h(s):= 2\beta(I) \ln_+\frac{2}{s} $, so that for any $\beta,I,H,f$ as above we have
\begin{equation}\label{2.1}
\int_{H} f(z) \exp \left( -  2 \int_{I} \ln |z-e^{i\tht}|  d\beta(\tht) \right) dz \leq  \int_{H} f(z) |z-e^{i\tht^*}|^{ - 2\beta(I)} dz.
\end{equation}	

Since Lemmas \ref{L.3.3}--\ref{L.3.5} are all stated at a single time $t$, we will drop $t$ from our notation in the proofs below.  Hence we will have $\omega(x), \Psi(x)$, and $R(z)$.  For $z\in\bbD$ we will also denote
\[
\Lambda(z) := \det D\calS(z) \, |\omega(\calS(z))|\ge 0.
\]
We note that
\beq\lb{6.0}
\int_\bbD \Lambda(z) dz\le \|\omega\|_{L^\infty} \int_\bbD \det D\calS(z)dz = |\Omega|\, \|\omega\|_{L^\infty},
\eeq
and that constants $C_1,C_2,\dots$ below will always be allowed to depend (only) on $\Omega$.

\subsection{Proof of Lemma \ref{L.3.4}}
%\begin{proof}[Proof of Lemma \ref{L.3.4}]
We have
	    \begin{align}\label{3.1}
	    	\frac{|\xi-z|^2}{|\xi-z^*|^2|z|^2} &=1-\frac{|\xi z-z^*z|^2-|\xi-z|^2}{|\xi-z^*|^2|z|^2} \notag \\
	    	&=1-\frac{(|\xi|^2|z|^2-2 {\rm Re}\, (\xi \bar z) +1)- (|\xi|^2- 2 {\rm Re}\,(\xi \bar z) +|z|^2)}{|\xi-z^*|^2|z|^2}  \notag \\
	    	&= 1-	\frac{(1-|\xi|^2)(1-|z|^2)}{|\xi-z^*|^2|z|^2}
	    \end{align}
for $\xi,z\in \mathbb{D}$ with $z\neq 0$,  which also means that  $\frac{|\xi-z|^2}{|\xi-z^*|^2|z|^2}\in (0,1)$ when  $z\neq 0,\xi$.
Hence
	\begin{align*}
	-\ln	\frac{|\xi-z|}{|\xi-z^*||z|} = -\frac{1}{2} \ln\left(1-\frac{(1-|\xi|^2)(1-|z|^2)}{|\xi-z^*|^2|z|^2}\right) 
	\geq \frac{1}{2} \frac{(1-|\xi|^2)(1-|z|^2)}{|\xi-z^*|^2|z|^2},
	\end{align*}
and so for each $\xi \in \mathbb{D}$ we have
	\begin{align*}
		\Psi(\calS(\xi)) &\geq \frac{1}{4\pi} \int_{\mathbb{D}} \frac{(1-|\xi|^2)(1-|z|^2)}{|\xi-z^*|^2|z|^2}  \Lambda(z)dz
%	    &= \frac{1}{4\pi} \int_{\mathbb{D}} \frac{(1-|\xi|)(1+|\xi|) (1-|z|)(1+|z|)  }{|\xi-z^*|^2|z|^2} \Lambda(z) dz \\
		\geq \frac{1-|\xi|}{4\pi}\int_{\mathbb{D}} \frac{(1-|z|)}{|\xi-z^*|^2|z|^2}  \Lambda(z) dz.
	\end{align*}
Given any $z,\xi\in\bbD$, let $M:=\max\{|z-\xi|, 1-|\xi|\}$.
Then $1-|z|\le 1-|\xi|+|z-\xi|\le 2M$, so
\[
|\xi-z^*| \, |z| \leq |\xi-z| + |z-z^*| \, |z|  = |z-\xi| + 1-|z|^2 \leq |z-\xi| + 2(1-|z|) \leq 5M
\]
when $z\neq 0$, and the result follows. 

\subsection{Proof of Lemma \ref{L.3.3}}
%\begin{proof}[Proof of Lemma \ref{L.3.3}]
Identity \eqref{3.1} and $-\ln(1-r) \leq (\frac{r}{1-r})^{\frac{1}{2}}$ for $r\in[0,1)$ (equality holds for $r=0$ and the right-hand side has a larger derivative on $(0,1)$) show that
	\begin{align*}
	- \ln 	\frac{|\xi-z|}{|\xi-z^*||z|} 
	%&= - \frac{1}{2} \ln 	\frac{|\xi-z|^2}{|\xi-z^*|^2|z|^2}  \\
	%&= - \frac{1}{2} \ln\left(1-\frac{(1-|\xi|^2)(1-|z|^2)}{|\xi-z^*|^2|z|^2}]\right) \\
	\leq \frac{1}{2} \left( \frac{ \frac{(1-|\xi|^2)(1-|z|^2)}{|\xi-z^*|^2|z|^2}    }{  \frac{|\xi-z|^2}{|\xi-z^*|^2|z|^2}  }   \right)^{\frac{1}{2}} 
%	&= \frac{1}{2}\frac{\sqrt{(1+|\xi|) (1+|z|)  (1-|\xi|)  (1-|z|)}}{|\xi-z|} \\
	\leq  \frac{(1-|\xi|)^{\frac{1}{2}} (1-|z|)^{\frac{1}{2}}}{|\xi-z|}.
	\end{align*}
Hence it suffices to show that there is $C_1 >0$ such that
\beq\lb{6.1}
\int_{\mathbb{D}} \frac{ (1-|z|)^{\frac{1}{2}} }{|z-\xi|} \Lambda(z) dz  \leq  C_1 \|\omega\|_{L^\infty} (1-|\xi|)^{2\hat{\alpha}-\frac 12},
\eeq
where $\hat{\alpha}: = \min\{1-\alpha_*, \frac 14\}$ (note that $2\hat{\alpha}-\frac 12\le 0$).  From \eqref{6.0} we see that it in fact suffices to replace $\bbD$ by $A_1: = B(\xi,\delta)\cap \bbD$   in \eqref{6.1}.

Let us decompose $A_1$ into $A_2: = B(\xi,\eps)\cap A_1 $ with  $\eps: = \frac{1-|\xi|}{2} $, $A_3: = B(\tilde{\xi},\eps) \cap A_1$ with $\tilde{\xi} := \frac{\xi}{|\xi|}$, and $A_4 := A_1 \backslash (A_2 \cup A_3)$. 
% Let $\tht_{\xi} := \arg(\xi) $ and  $\alpha_{\xi} := \beta^+_{d} (\tht_{\xi}+\delta) - \beta^+_{d} (\tht_{\xi}-\delta)$. 
Now  Lemma \ref{L.3.1} and \eqref{2.1} with $H:=A_1$, $I:=(\arg(\xi)-2\delta,\arg(\xi)+2\delta)$,  $f(z):=\frac{(1-|z|)^{1/2 }}{|z-\xi|}$, and $\beta:=\sum_{\tht_j \in I} \alpha_j^+ \delta_{\theta_j}$, where $\delta_{\theta_j}$ is the Dirac mass at $\theta_j$, yield
    \begin{align*}
    	\int_{A_1} \frac{(1-|z|)^{\frac{1}{2}}}{|z-\xi|} \Lambda(z) & dz  \leq C_2 ||\omega||_{L^\infty} \int_{A_1} \frac{(1-|z|)^{\frac{1}{2} }}{|z-\xi|} \Pi_{\tht_j \in I} |z-e^{i\theta_j}|^{-2\alpha^+_j} dz \\ 
	\leq & C_2 ||\omega||_{L^\infty} \int_{A_1} \frac{(1-|z|)^{\frac{1}{2} }}{|z-\xi|} |z-\tilde{\xi}|^{-2\alpha_*} dz \\
 %   	\leq &C_6\int_{B_{\xi}} \frac{(1-|z|)^{\frac{1}{2} }}{|z-\xi|} |z-\tilde{\xi}|^{-2\alpha_*} dz \\
    	\leq &C_2 ||\omega||_{L^\infty} \left( \int_{A_2} \frac{\eps^{\frac 12-2\alpha_*}}{ |z-\xi|}dz + \int_{A_3} \frac {|z-\tilde{\xi}|^{\frac12 -2\alpha_*}} {\eps}dz  + \int_{A_4} 3^3  |z-\xi|^{-\frac 12-2\alpha_*}dz \right) \\
    	\leq &C_3 ||\omega||_{L^\infty} \eps^{\frac 32-2\alpha_*} 
    	\leq  2C_3 ||\omega||_{L^\infty} (1-|\xi|)^{2\hat{\alpha}-\frac 12} 
    \end{align*}
because \eqref{1.11}  shows that $\sum_{\tht_j \in I} \alpha^+_j \le \alpha_*<1$.  
This therefore finishes the proof of \eqref{6.1}.

\subsection{Proof of Lemma \ref{L.3.5}}
%\begin{proof}[Proof of Lemma \ref{L.3.5}]
%   Let $\tilde{\xi} = \frac{\xi}{|\xi|}$, and 
First integrate  over $A_0 : = \mathbb{D} \backslash  B(\xi,\delta)$.  Then \eqref{2.3}, \eqref{6.0}, and 
%$|\tilde z-\tilde z^*|\le 2|z-\tilde z^*|$ for all $z,\tilde z\in\bbD$ 
\beq\lb{6.3}
|z-\tilde{z}^*| \ge |\tilde{z}^*|-1 \ge \frac {|\tilde{z}-\tilde{z}^*| }2 \ge 1-|\tilde z|
\eeq
for any $z,\tilde z\in\bbD$ yield
%, and we first consider integration over $D_1$, we have
    \begin{align*}
    	\int_{A_0} \frac{|R(z)|}{|z-\xi||z-\xi^*|} dz 
	%&\leq \frac{1}{\delta^2} 	\int_{D_1} |R(z)| dz \\
    	&\leq \frac 1{\delta^2} 	\int_{A_0}   \int_{\mathbb{D}} \frac{|\tilde{z}-\tilde{z}^*|}{|z-\tilde{z} |  |z-\tilde{z}^*|} \Lambda(\tilde{z}) d\tilde{z}dz \\
    	&\leq   \frac 2{\delta^2}  \int_{\mathbb{D}}\int_{A_0} \frac{dz}{|z-\tilde{z} | }  \Lambda(\tilde{z}) d\tilde{z} \\
 %   	&=  3C_{8}||\omega_0||_{L^\infty} \int_{\mathbb{D}}  \int_{D_1} \frac{1}{|\tilde{z}-z | } | \det D\calS(\tilde{z}) dzd\tilde{z} \\
    	&\leq   \frac {4\pi}{\delta^2}    \int_{\mathbb{D}} \Lambda(\tilde{z}) d\tilde{z} \\
	& = \frac{4\pi  |\Omega|}{\delta^2} \, ||\omega||_{L^\infty} .
%    	&= 12\pi C_{8}||\omega_0||_{L^\infty}  (\int_{B(0,1-\frac{\delta}{8})}  det D\calS(z) dz + \int_{\mathbb{D} \backslash B(0,1-\frac{\delta}{8})}  \det D\calS(z) dz ) \\
%    	&\leq ||\omega_0||_{L^\infty}(C_{9}  + 12\pi C_{8} \int_{\mathbb{D} \backslash B(0,1-\frac{\delta}{8})}  \det D\calS(z) dz),\\
    \end{align*}
%    where the last inequality is by theorem 0.1, 0.2 and the region (which makes $\det\calS(z) \leq (\frac{\delta}{8})^{\sum_{i} |\alpha_i| }$). For the second integration, the region is finitely covered by $N_{\delta}$ many balls $B_k$ whose radius is $\frac{\delta}{4}$ and centered at some $b_k \in \partial\mathbb{D}$, and let $\alpha^k = \beta^+_{d} (\arg b_k+2\delta) - \beta^+_{d} (\arg b_k-2\delta)$. Further by theorem 0.1, 0.2 and folding lemma (same explanation as before with $f =1, g = 0, h = 2\beta(I)\ln_+(\frac{2}{s})$), we have 
%    \begin{align*}
%    \int_{\mathbb{D} \backslash B(0,1-\frac{\delta}{8})}  \det D\calS(z) dz&\leq \sum_{n = 1}^{N_{\delta}} \int_{B_n} \det DS(z) dz \\
%    &\leq C_{10}\sum_{n = 1}^{N_{\delta}} \int_{B_n} |z-b_1|^{-2\alpha^n} dz  \\
%    &\leq  C_{10}\sum_{n = 1}^{N_{\delta}} \int_{B_n} |z-b_1|^{-2\alpha_*} dz \leq C_{11}
%    \end{align*}
%    Hence we have 
%    \beq\lb{R1}
%    \int_{D_1} \frac{|R(z)|}{|z-\xi||z-\xi^*|} dz \leq C_{12}||\omega_0||_{L^\infty}
%    \eeq  
So it remains to integrate over $A_1:=B(\xi,\delta)\cap\bbD$.   From \eqref{6.3}, $|\xi|-\delta\ge \frac 58$, and \eqref{6.0} we have
%, and boundedness of $\det D\calS(\tilde z) $ on $B(0,\frac 12)$ (see Lemma \ref{L.3.1}) yield
    \beq \lb{6.2}
    	\int_{A_1} \frac{1}{|z-\xi||z-\xi^*|}   \int_{B(0,1/2)} \frac{|\tilde{z}-\tilde{z}^*|}{|z-\tilde{z} |  |z-\tilde{z}^*|  } \Lambda(\tilde{z})  d\tilde{z}dz \le  C_{1} |\ln (1-|\xi|)| \, ||\omega||_{L^{\infty}},
%    	&\leq C_{1} ||\omega_0||_{L^{\infty}}	\int_{A_1} \frac{1}{|z-\xi||z-\xi^*|} dz\\
%    	&\leq C_{14} ||\omega_0||_{L^{\infty}}|\ln(1-|\xi|)|
    \eeq
 where we also used that  with $B_\xi := B(\xi,\frac{|\xi-\xi'|}{2})\cap\bbD$ and $B_{\xi'} := B(\xi',\frac{|\xi-\xi'|}{2})\cap\bbD$ we have
    \beq \lb{6.4}
  	\int_{\mathbb{D}} \frac{dz}{|z-\xi||z-\xi'|} 
	%	\int_{\mathbb{D} \backslash B_y} \frac{1}{|z-y||z-y^*|}dz +	\int_{ B_y} \frac{1}{|z-y||z-y^*|}dz \\
  	\leq  3\int_{\mathbb{D} \setminus (B_\xi\cup B_{\xi'}) } \frac{dz}{|z-\xi|^2} +	\frac{4}{|\xi-\xi'|} \int_{ B_\xi} \frac{dz}{|z-\xi|} \le 6\pi \ln_+ \frac 1{|\xi-\xi'|}+50
\eeq
  for any $\xi,\xi'\in\bbC$.
  
We now let $\eps: = 1-|\xi| $ and split $A_1$ into $A_2:=  B(\xi,\frac{\eps}{4})$ and $A_3:=A_1\setminus A_2$.  We start with $A_2$, and let $E_1 := B(\xi,\frac\eps 2)$ and $E_2 := \mathbb{D} \backslash (B(0,\frac{1}{2}) \cup B(\xi,\frac\eps 2))$.  We also denote $M(\xi,z) := \max\{|z-\xi|, 1-|\xi|\}$.  When $(z,\tilde z)\in A_2\times E_1$, then \eqref{6.3}, $|z-\xi^*|\ge\eps$, and     \eqref{6.4} show that
    \begin{align*}
    		\int_{A_2} \frac{1}{|z-\xi||z-\xi^*|} \int_{E_1} \frac{|\tilde{z}-\tilde{z}^*|}{|z-\tilde{z} |  |z-\tilde{z}^*|} \Lambda(\tilde{z}) d\tilde{z}dz 
    		& \le \frac 2{\eps} \int_{E_1} \Lambda(\tilde{z}) \int_{A_2} \frac{ dz}{|z-\xi||z-\tilde{z} |}   d\tilde{z} \\
    	    & \leq \frac{C_{2}}\eps \int_{E_1} \Lambda(\tilde{z}) \, |\ln |\tilde{z}-\xi| | \, d\tilde{z}.
    	    %\leq C_{18} & \frac{-\det D\calS(\xi) }{\eps}\int_{E_1} \omega(\calS(\tilde{z}))\ln(|\tilde{z}-\xi|) d\tilde{z} \\
    		%\leq C||\omega_0||_{L^\infty}&\int_{D_2} \frac{1}{|z-\xi||z-\xi^*|}  \int_{E_1} \frac{1}{|\tilde{z}-\xi |} | det\calS(\tilde{z}) d\tilde{z}dz \\
    		%\leq C||\omega_0||_{L^\infty}&\int_{D_2} \frac{1}{|z-\xi||z-\xi^*|}  \int_{E_1} \frac{|\tilde{z}-\tilde{z}^*|}{|\tilde{z}-\xi ||\xi -\tilde{z}^*|} | det\calS(\tilde{z}) d\tilde{z}dz \\
    \end{align*}
%    
%    Note that if $\tilde{z} \notin B(0,\frac{1}{2})$, then we have $1-|z| \geq \frac{|\tilde{z} - \tilde{z}^*|}{5}$.
%    
%    
%    where 
%    \begin{enumerate}
%    	\item 
%    	first inequality is because by definition of $D_2,E_1$, we have $|\tilde{z}-\tilde{z}^*| \sim 1-|\tilde{z}| \sim |z-\xi^*| \sim |\tilde{z}^*-\xi| \sim |z-\tilde{z}^*| \sim \eps$ for all $z \in D_2, \tilde{z} \in E_1$, and $|\tilde{z} - \xi| < \eps$
%    	\item
%    	First equality is by Fubini Theorem,
%    	\item 
%    	Second inequality is by $|\tilde{z}| < 1$ and the following computation: for any fixed $\xi \neq \tilde{z} \in E_1$, we have 
%    	\begin{align*}
%    		\int_{D_2} \frac{1}{|z-\xi||z-\tilde{z}|}dz &= \int_{B(\xi,\frac{|\tilde{z}-\xi|}{2})} +\int_{B(\tilde{z},\frac{|\tilde{z}-\xi|}{2})} +\int_{D_2 \backslash (B(\xi,\frac{|\tilde{z}-\xi|}{2}) \cup B(\tilde{z},\frac{|\tilde{z}-\xi|}{2}) } \frac{1}{|z-\xi||z-\tilde{z}|}dz\\
%    		&\leq \frac{2}{|\tilde{z}-\xi|} \left(\int_{B(\xi,\frac{|\tilde{z}-\xi|}{2})} \frac{1}{|z-\xi|} +\int_{B(\tilde{z},\frac{|\tilde{z}-\xi|}{2})} \frac{1}{|z-\tilde{z}|} \right) \\
%    		&+ 2 \int_{D_2 \backslash (B(\xi,\frac{|\tilde{z}-\xi|}{2}) \cup B(\tilde{z},\frac{|\tilde{z}-\xi|}{2}) } \frac{1}{|z-\xi|^2}+\frac{1}{|z-\tilde{z}|^2}dz\\
%    		&\leq \frac{2}{|\xi-\tilde{z}|} 4\pi|\xi-\tilde{z}| + 8\pi|\ln|\tilde{z}-\xi|| 
%    	\end{align*}
%    	and hence 
%    	\beq\lb{dLog}
%    	\int_{D_2} \frac{1}{|z-\xi||z-\tilde{z}|}dz \leq -9\pi\ln|\xi-\tilde{z}|
%    	\eeq 
%    \end{enumerate}
From Lemma \ref{L.3.1} and \eqref{1.11} we see that $\det D\calS(\tilde{z}) \leq C_3(1-|\tilde z|)^{-2}$ for some $C_3$ and  all $\tilde{z} \in \bbD$, hence
\[
\int_{B(\xi,\eps^2)} \Lambda(\tilde{z}) \, |\ln |\tilde{z}-\xi| | \, d\tilde{z} \le 4C_3\eps^{-2} ||\omega||_{L^\infty}\int_{B(\xi,\eps^2)}  |\ln |\tilde{z}-\xi| | \, d\tilde{z} \le C_4 ||\omega||_{L^\infty} \eps^2 |\ln\eps|.
\]
From the last two estimates and $M(\xi,\tilde z) =\eps\le 2(1-|\tilde z|)$ for $\tilde z\in E_1$ it now follows that
\begin{align*}
\int_{A_2} \frac{1}{|z-\xi||z-\xi^*|} \int_{E_1} \frac{|\tilde{z}-\tilde{z}^*|}{|z-\tilde{z} |  |z-\tilde{z}^*|} & \Lambda(\tilde{z}) d\tilde{z}dz   \le C_2C_4 ||\omega||_{L^\infty} \eps |\ln\eps| + \frac{C_{2}|\ln \eps | }\eps \int_{E_1} \Lambda(\tilde{z}) d\tilde{z}
\\ & \le C_5  |\ln \eps|\left(\int_{E_1}\frac{1-|\tilde{z}|}{M(\xi,\tilde{z})^2} \Lambda(\tilde{z}) d\tilde{z}+||\omega||_{L^\infty} \right).
\end{align*}
%
% so let's define $A_1 = B(\xi,\eps^2)$ and $A_2 = E_1 \backslash A_1$. We have 
%	\begin{align*}
%		\int_{A_1}\frac{\Lambda(\tilde{z})(1-|\tilde{z}|)}{M(\xi,\tilde{z})^2}\ln(|\tilde{z}-\xi|) d\tilde{z} &\leq C_{17} \eps^{-2} ||\omega_0||_{L^\infty}   \int_{A_1}  -\ln(|\tilde{z}-\xi|) d\tilde{z}  \\
%		&= C_{17} \eps^{-3}||\omega_0||_{L^\infty}   \int_{0}^{2\pi}\int_{0}^{\eps^2}  -r\ln(r) dxd\tht \\
%		&=C_{17} \eps^{-3}2\pi||\omega_0||_{L^\infty} \frac{-1}{4}\eps^4(4\ln(\eps)-1)\\
%		&\leq C_{18}||\omega_0||_{L^\infty} |\ln \eps|\eps
%	\end{align*}
%    and 
%    \begin{align*}
%    \int_{A_2}\frac{\Lambda(\tilde{z})(1-|\tilde{z}|)}{M(\xi,\tilde{z})^2}\ln(|\tilde{z}-\xi|) d\tilde{z} &\leq 2|\ln(\eps) |\int_{A_2} \frac{\Lambda(\tilde{z})(1-|\tilde{z}|)}{M(\xi,\tilde{z})^2} d\tilde{z}.
%    \end{align*}
%    Combining these we can conclude that 
%    \beq\lb{R2}
%    	\int_{D_2}\int_{E_1} \frac{1}{|z-\xi||z-\xi^*|}  \frac{\Lambda(\tilde{z})|\tilde{z}-\tilde{z}^*|}{|\tilde{z}-z |  |z-\tilde{z}^*|} |  d\tilde{z}dz \leq  C_{19}|\ln(\eps)|\left(\int_{B_{\eps}}\frac{\Lambda(\tilde{z})(1-|\tilde{z}|)}{M(\xi,\tilde{z})^2} d\tilde{z}+||\omega_0||_{L^\infty} \right)
%    \eeq

Moreover, for all $(z,\tilde z) \in A_2\times E_2$ we have $|z-\tilde{z}^*|\ge |z-\tilde{z}| \geq  \frac{|\tilde{z}-\xi|}2 \geq \frac{1-|\xi|}{4}$ and $|\tilde z-\tilde{z}^*| \leq 3(1-|\tilde{z}|)$, therefore
   \begin{align*}
   	\int_{A_2} \frac{1}{|z-\xi||z-\xi^*|}  \int_{E_2} \frac{|\tilde{z}-\tilde{z}^*|}{|z-\tilde{z} |  |z-\tilde{z}^*|} \Lambda(\tilde{z}) d\tilde{z}dz
%   	&\leq C_{20}\int_{A_2} \frac{1}{|z-\xi||z-\xi^*|}  \int_{E_2} \frac{\Lambda(\tilde{z})(1-|\tilde{z}|)}{M(\xi,\tilde{z})^2  } |  d\tilde{z}dz\\
   	&\le 48 \int_{E_2}  \frac{1-|\tilde{z}|}{M(\xi,\tilde{z})^2  } \Lambda(\tilde{z})  d\tilde z \int_{A_2} \frac{dz}{|z-\xi||z-\xi^*|}\\
   	&\le C_6 |\ln(1-|\xi|)| \int_{E_2}  \frac{1-|\tilde{z}|}{M(\xi,\tilde{z})^2  } \Lambda(\tilde{z}) d\tilde z,
   \end{align*}   		
    %For $D_4 = B_{\xi} \backslash B(\xi,\frac{\eps}{4})$, 
    %and hence combining both we get the integration inequality (stated in the lemma) is true if we only integrate over $D_2$.
    where we also used \eqref{6.4}. 
    The last two estimates and \eqref{6.2}  show that
    \[
    \int_{A_2} \frac{|R(z)|}{|z-\xi||z-\xi^*|} dz \le 
    (C_1+C_5+C_6)  |\ln(1-|\xi|)|\left(\int_{\bbD}\frac{1-|z|}{M(\xi,z)^2} \Lambda(z) dz +||\omega||_{L^\infty} \right),
    \]
   so it remains to integrate over $A_3$.

Let $F_1 := B(\xi, \frac{\eps}{8})$,  $F_2 := \mathbb{D} \setminus ( B(0,\frac{1}{2}) \cup B(\xi,2\delta))$, and $F_3 := (B(\xi,2\delta) \cap\bbD) \setminus B(\xi, \frac{\eps}{8})$.
%(B(0,\frac{1}{2})\cup F_1) $, 
Then for all $(z,\tilde z) \in A_3\times F_1$ we have $|z-\tilde{z}^*|\ge |z-\tilde{z}| \geq \frac{1-|\xi|}{8} \geq  |\tilde z-\xi|$ and $|\tilde z-\tilde{z}^*| \leq 3(1-|\tilde{z}|)$, which together with \eqref{6.4} yields
  	\begin{align*}
  	\int_{A_3} \frac{1}{|z-\xi||z-\xi^*|}  \int_{F_1}  \frac{|\tilde{z}-\tilde{z}^*|}{|z-\tilde{z} |  |z-\tilde{z}^*|}  \Lambda(\tilde{z}) d\tilde{z}dz
  	&\leq 192  \int_{F_1} \frac{1-|\tilde{z}|}{M(\xi,\tilde{z})^2}  \Lambda(\tilde{z}) d\tilde{z}  \int_{A_3} \frac{dz}{|z-\xi^*||z-\xi|}  \\
%  	&= C_{24} \int_{B_1} \Lambda(\tilde{z}) \frac{1-|\tilde{z}|}{M(\xi,\tilde{z})^2}  \int_{D_3} \frac{1}{|z-\xi^*||z-\xi|}     d\tilde{z}\\
  	&\leq C_7 |\ln(1-|\xi|)| \int_{F_1}  \frac{1-|\tilde{z}|}{M(\xi,\tilde{z})^2}  \Lambda(\tilde{z})  d\tilde{z}.
  	\end{align*}
%	where 
%	\begin{enumerate}
%		\item The first inequality is by $|\tilde{z} - \tilde{z}^*| \leq 5(1-|\tilde{z}|)$, $|z-\tilde{z}^*| \geq \frac{1-|x|}{2}$, $|\tilde{z}-\xi| < \frac{1-|\xi|}{8}$ and $|z - \tilde{z}| \geq \frac{1-|\xi|}{8} $,
%
%		\item Second inequality is by computation at last.
%	\end{enumerate}	
And from \eqref{6.3}, \eqref{6.0}, and \eqref{6.4} we obtain
%	For $(z,\tilde z) \in A_3\times F_2$ we have from $|\tilde{z}-\tilde{z}^*|\le 2 |z-\tilde{z}^*|$ and \eqref{6.3} that
%	
%Similar as before, by the choice of $E_4$ (we have $|\tilde{z}-\tilde{z}^*| < 5$  and $|\tilde{z}-z|, |\tilde{z}^*-z|\geq \delta$ for all $z \in D_3, \tilde{z} \in E_4$):
   \begin{align*}
   	\int_{A_3} \frac{1}{|z-\xi||z-\xi^*|}  \int_{F_2}\frac{|\tilde{z}-\tilde{z}^*|}{|z-\tilde{z}|  |z-\tilde{z}^*|}  \Lambda(\tilde{z})  d\tilde{z}dz
%   	& \leq C_{21}\int_{D_3} \frac{1}{|z-\xi||z-\xi^*|}  \int_{E_4}\Lambda(\tilde{z}) d\tilde{z}dz \\
   	&\le \frac 2\delta \int_{F_2} \Lambda(\tilde{z}) d\tilde{z} \int_{A_3} \frac{dz}{|z-\xi||z-\xi^*|}   \\
   	&\leq C_8   |\ln(1-|\xi|)| \,||\omega||_{L^{\infty}}.
%	 |\int_{\mathbb{D}}\Lambda(\tilde{z})d\tilde{z}\\
%   	&\leq C_{23}|||\omega_0||_{L^{\infty}}\ln(1-|\xi|) |\\
   \end{align*}
   
For the integral involving $(z,\tilde z)\in A_3\times F_3$, let  $F_4 :=  F_3 \cap  B(0,1-\eps^{\frac{1}{1-\alpha_*}})$ %where $\gamma := \max\{\alpha_*,  \frac 9{10} \}$, 
and for $\tilde z\in F_3$ let $A_{\tilde{z}} :=B(\tilde{z},\frac{|\tilde{z}-\xi|}{2}) \cap A_3$.  
%Since $|w|\le|z|$ implies 
%\beq\lb{7.11}
%|z^*-w^*| \le \left|\frac {z}{|{z}|^2} - \frac w{|{z}|^2} \right| + |w| \left|\frac 1{|{z}|^2} - \frac 1{|w|^2} \right| 
%\le \frac {2|w|^2+|w||z| }{|{z}|^2|w|^2} |{z}-w| \le \frac 3{|z|\,|w|} |{z}-w|,
%\eeq
From $|\xi|,|{\tilde z}|\ge \frac 12$ and \eqref{2.3} we get
\[
|\tilde{z}-\xi^*|\le |\tilde{z}-\tilde z^*| + 4|\tilde{z}-\xi| \le |\tilde{z}-\tilde z^*| + 8|\tilde{z}-z| \le 10| z-\tilde z^*|
\]
when also $z \notin  A_{\tilde{z}}$.  This, \eqref{6.4},  $|\tilde{z}-\tilde{z}^*| \leq 3(1-|\tilde{z}|)$, and $|\tilde{z} - \xi^*| \ge |\tilde{z}-\xi| \geq \frac{1-|\xi|}{8}$ for $\tilde z\in F_3$, and $|\tilde{z} - \xi^*| \geq \frac{|\tilde{z}-\xi|}{2}$ show that
  \begin{align*}
\int_{A_3} \frac{1}{|z-\xi||z-\xi^*|}  \int_{F_4} & \frac{|\tilde{z}-\tilde{z}^*|}{|z-\tilde{z} | |z-\tilde{z}^*|} \Lambda(\tilde{z}) d\tilde{z} dz \\
%&= \Big( \int_{E_3 \backslash E_5}\int_{D_3(\tilde{z})} \frac{1}{|z-\xi||z-\xi^*|}    \Lambda(\tilde{z})\frac{|\tilde{z}-\tilde{z}^*|}{|\tilde{z}-z | |z-\tilde{z}^*|}  dzd\tilde{z} \\
%&+  \int_{E_3 \backslash E_5}\int_{D_3 \backslash D_3(\tilde{z})} \frac{1}{|z-\xi||z-\xi^*|}    \Lambda(\tilde{z})\frac{|\tilde{z}-\tilde{z}^*|}{|\tilde{z}-z | |z-\tilde{z}^*|}  dzd\tilde{z} \Big)\\
 \leq  & 4 \int_{F_4}\int_{A_{\tilde{z}}} \frac{1}{|\tilde{z}-\xi||\tilde{z}-\xi^*|}    \frac{|\tilde{z}-\tilde{z}^*|}{|z-\tilde{z} | |z-\tilde{z}^*|}  \Lambda(\tilde{z}) dzd\tilde{z} \\
& + 20 \int_{F_4}\int_{A_3 \setminus A_{\tilde{z}}} \frac{1}{|z-\xi||z-\xi^*|}  \frac{|\tilde{z}-\tilde{z}^*|}{|\tilde{z}-\xi| |\tilde{z}-\xi^*|}   \Lambda(\tilde{z}) dzd\tilde{z} \\
\leq & C_9 \int_{F_4} \frac{|\tilde{z}-\tilde{z}^*|}{|\tilde{z}-\xi||\tilde{z}-\xi^*|}    \Lambda(\tilde{z})\left(|\ln(1-|\tilde{z}|)| + |\ln(1-|\xi|)|\right)d\tilde{z} \\
\leq & C_{10} |\ln(1-|\xi|)| \int_{F_4} \frac{1-|\tilde{z}|}{M(\xi,\tilde{z})^2}    \Lambda(\tilde{z})d\tilde{z}.
% \\
%&\leq C_{10}|\ln(1-|\xi|)| \int_{F_4} \frac{1-|\tilde{z}|}{M(\xi,\tilde{z})^2} \Lambda(\tilde{z}) d\tilde{z}
\end{align*}
%    where 
%    \begin{enumerate}
%%    	\item  First inequality is by the definition of $D_3(\tilde{z})$ (fixed $\tilde{z}$ first but the coefficients later don't depend on it), we have $\frac{1}{2|\tilde{z} -\xi|} \leq |z-\xi| \leq 2  |\tilde{z} -\xi|$ and $\frac{1}{3}|\tilde{z}-\xi^*| \leq |z-\xi^*| \leq 3 |\tilde{z}-\xi^*|$ for all $z \in D_3(\tilde{z})$, as well as $2|z-\tilde{z}| \geq |\xi-\tilde{z}|$ and $|\tilde{z}-\xi^*| \leq 3|\tilde{z}-\xi| \leq 6|\tilde{z}-z| \leq 12|z-\tilde{z}^*|$ for all $z \in D_3\backslash D_3(\tilde{z})$.
%%    	\item Second inequality is using the computation at the end of proof.
%%    	\item Third inequality is using the fact that $\tilde{z} \in E_3 \backslash E_5$ which means $(1-|\tilde{z}|) \geq 2\eps^{\frac{2}{1-\alpha}}$.
%    	\item Last inequality is by $|\tilde{z}-\tilde{z^*}| \leq 5(1-|\tilde{z}|)$, $|\tilde{z}-\xi| \geq \frac{1-|\xi|}{8}$, and $|\tilde{z} - \xi^*| \geq \frac{|\tilde{z}-\xi|}{2}$.
%    \end{enumerate}
    
 Finally, let $F_5:=F_3\setminus F_4$.  From \eqref{6.3}, \eqref{6.4}, Lemma \ref{L.3.1}, and \eqref{2.1} with $H:=F_5$, $I:=(\arg(\xi)-3\delta,\arg(\xi)+3\delta)$,  $f\equiv 1$, and $\beta:=\sum_{\tht_j \in I} \alpha_j^+ \delta_{\theta_j}$ we obtain
    \begin{align*}
    \int_{A_3} \frac{1}{|z-\xi||z-\xi^*|} & \int_{F_5}  \frac{|\tilde{z}-\tilde{z}^*|}{|z-\tilde{z} | |z-\tilde{z}^*|} \Lambda(\tilde{z}) d\tilde{z}     dz \\
     &\leq 2 \int_{F_5} \Lambda(\tilde{z}) \int_{A_3}\frac{dz}{|z-\xi||z-\xi^*| |z-\tilde{z}|} d\tilde{z}   \\
    &\leq 8 \int_{ F_5}  \Lambda(\tilde{z}) \left(\int_{A_{\tilde{z}} } \frac{dz}{|\tilde{z}-\xi|^2|\tilde{z}-z | } + \int_{A_3\setminus A_{\tilde{z}}   }\frac{dz}{|z-\xi||z-\xi^*||\tilde{z}-\xi|}     \right) d\tilde{z} \\
    &\leq C_{11} \int_{ F_5} \Lambda(\tilde{z}) \left( \frac{1}{|\tilde{z}-\xi|}  +   \frac{ |\ln(1-|\xi|)|}{|\tilde{z}-\xi|}   \right)  d\tilde{z} \\
    &\leq C_{12} \frac{ |\ln(1-|\xi|)|}{\eps}  ||\omega||_{L^{\infty}} \int_{F_5}   \Pi_{\tht_j \in I} |\tilde z-e^{i\theta_j}|^{-2\alpha_j^+}d\tilde{z} \\
    &\leq C_{13} \frac{ |\ln(1-|\xi|)|}{\eps}  ||\omega||_{L^{\infty}} \int_{F_5} \left|\tilde{z}-\frac{\xi}{|\xi|} \right|^{-2\alpha_*}d\tilde{z} \\
    &\leq C_{14}|\ln(1-|\xi|)| \, ||\omega||_{L^{\infty}},
     \end{align*}
     where in the last inequality we used that $|F_5|\le \eps^{\frac 1{1-\alpha_*}}$, which is less than the area of a disc with radius $\eps^{\frac 1{2-2\alpha_*}}$.
%    
%    
%     
%     let $B^*$ to be a ball centered at $\frac{\xi}{|\xi|}$ whose area equals the area of $E_5$ so that the radius of $B^*$ is less than $ 10 \eps^{\frac{1}{1-\alpha}}$ ; 
%let $\alpha_{\xi }  = \beta^+_{d}(\arg(\xi) + 2\delta) - \beta^+_{d}(\arg(\xi) - 2\delta)$
%
%
%  where 
%  \begin{enumerate}
%%  \item First inequality is by $\frac{|\tilde{z}-\tilde{z}^*|}{|\tilde{z}-z||z-\tilde{z}^*|} \leq \frac{1}{|z-\tilde{z}|} + \frac{1}{|z-\tilde{z}^*|} \leq \frac{3}{|z-\tilde{z}|}$ for all $z \in D_3$ and $\tilde{z} \in E_5$
%%  
%%  \item Second inequality is by any fixed $\tilde{z} \in E_5$ and the definition of $D_3(\tilde{z})$ ensures that if $z \in D_3(\tilde{z})$ then we have $|z-\xi|, |z-\xi^*| \geq \frac{1}{4} |\tilde{z}-\xi|$, and if $z \in D_3 \backslash D_3(\tilde{z}) $ we have $|z-\tilde{z}| \geq \frac{|\tilde{z}-\xi|}{2}$,
%%  \item Third inequality is by direct computation and the computation at the end of proof.
%%  
%  \item Fourth is straightforward and sixth is by folding lemma by letting $f =1, g = 0, h = \frac{2}{\pi} \ln(\frac{1}{x})$ for $x > 1$.
%  
%  \item Last inequality is by replacing $E_5$ with $B^*$ and two have the same size, but the integrand will become larger. The integration over $B^*$ is bounded by
%  $$C\int_{0}^{\eps^{\frac{1}{1-\alpha}}} r^{1-2\alpha_*}dr \leq C\eps^2, $$ by the choice of $\alpha$.
%  \end{enumerate}
Combining the above estimates and \eqref{6.2} yields
\[
    \int_{A_3} \frac{|R(z)|}{|z-\xi||z-\xi^*|} dz \le 
    (C_1+C_7+C_8+C_{10}+C_{14})  |\ln(1-|\xi|)|\left(\int_{\bbD}\frac{1-|z|}{M(\xi,z)^2} \Lambda(z) dz +||\omega||_{L^\infty} \right),
\]
and the result follows. 
%\end{proof}

%Below is the computation mentioned above.Here we just showed a stronger version of inequality that could be applied to every place we used it before. We'll show that for all $y \in \mathbb{D}$, we have
%  $$\int_{\mathbb{D}} \frac{1}{|z-y||z-y^*|}dz \leq 4\pi \ln(1-|y|).$$
%  
%  To show this, let $B_y := B(y,\frac{1-|y|}{2})$, and we have $|z-y| \leq 2 |z-y^*|$ for all $z \in \mathbb{D}\backslash B_y$, hence we have 
%  \begin{align*}
%  	\int_{\mathbb{D}} \frac{1}{|z-y||z-y^*|}dz &=	\int_{\mathbb{D} \backslash B_y} \frac{1}{|z-y||z-y^*|}dz +	\int_{ B_y} \frac{1}{|z-y||z-y^*|}dz \\
%  	&\leq 2\left(	\int_{\mathbb{D} \backslash B_y} \frac{1}{|z-y|^2}dz +	\frac{1}{1-|y|}\int_{ B_y} \frac{1}{|z-y|}dz\right)\\
%  	&\leq 4\pi|\ln(1-|y|) |   	
%  \end{align*}
%  

\section{Proof of Lemma \ref{L.3.6}}

We see from \eqref{2.2}, a change of variables in the integral from \eqref{7.13}, and \eqref{7.18}  that we need to show boundedness and continuity of $R$ and 
\[
Q(t,\xi) :=  \int_{\bbD}  \left(\frac{z-\xi}{|z-\xi|^2}  - \frac{z-\xi^*}{|z-\xi^*|^2}\right) \cdot R(t,z)  \, \omega(t,\calS(z))  dz
\]
on $[0,\infty)\times K$ for any compact $K\subseteq\bbD$, as well as that $\partial_t \Psi(t,x)= - \frac 1{2\pi}Q(t,\calT(x))$ holds for each $(t,x)\in[0,\infty)\times\Omega$.
%
%	For the completeness and by $\calT$ is continuous, we'll show for any compact subset $K$,  $R(t,\xi), I(t,\xi) \in C([0,T] \times K)$ as  well as the integrand of $I$ could by bounded by some $Cf(x,z) $ where $f(x,\cdot) \in L^1(\Omega)$ uniformly on $[0,T]$ so that we could pass the partial t inside. 
%	%For simplicity of notation, we'll denote $f(\xi,z) = \left(\frac{z-\xi}{|z-\calT(\xi)|^2}  - \frac{z-\calT(\xi)^*}{|z-\calT(\xi)^*|^2}\right) $.
%	
%	Remark: when $\xi \in K$, then the continuity of $R$ is the same as continuity of $u$ on $[0,T] \times K$.
%	

So fix any such $K$	 and let $d :=  \text{dist}(K, \partial \mathbb{D}) > 0$, then fix any $(t,\xi) \in [0,\infty) \times K$ and let $B: =  B(\xi,\frac{d}{2})$ and $B' := \overline{ B(\xi,\frac{d}{4})}$. With $C_d:=\sup_{|z|\le 1-d/2} \det D\calS(z)$, and using  \eqref{2.3}, $|w-z^*|\ge|w-z|$ for all $z,w\in\bbD$, \eqref{6.4}, and \eqref{6.0}, we obtain for any  $(t',\xi') \in [0,\infty) \times B'$,
\begin{align*}
| R(t,\xi) - 	R(t,\xi')|	
%&\leq \int_{\mathbb{D}} \left(\frac{|\xi-\xi'|}{|\xi-z||\xi'-z|}    + \frac{|\xi-\xi'|}{|\xi-z^*||\xi'-z^*|} \right) \det D\calS(z) ||\omega_0||_{L^\infty} dz \\
\leq 
%& C ||\omega||_{L^\infty} \int_{B}\left(\frac{|\xi-\xi'|}{|\xi-z||\xi'-z|}    + \frac{|\xi-\xi'|}{|\xi-z^*||\xi'-z^*|} \right)  dz \\
		&  ||\omega||_{L^\infty} \left( \int_{B} + \int_{\mathbb{D} \backslash B} \right)\left(\frac{|\xi-\xi'|}{|\xi-z||\xi'-z|}    + \frac{|\xi-\xi'|}{|\xi-z^*||\xi'-z^*|} \right)\det D\calS(z) dz \\
		\leq & 2 ||\omega||_{L^\infty} |\xi-\xi'| \left(6\pi C_d\ln_+\frac{1}{|\xi-\xi'|} +50C_d+\frac {8|\Omega|}{d^2} \right)\\
\end{align*}
and (using also $|z-z^*|\le 2|\xi'-z^*|$ and H\" older's inequality)
\begin{align} 
 |R(t,\xi') - 	R(t',\xi')|	
 \leq &  \int_{\mathbb{D}} \frac{|z-z^*|}{|\xi'-z||\xi'-z^*|} \det D\calS(z) |\omega(t,\calS(z)) - \omega(t',\calS(z)) | dz  \notag\\	
%		\leq & 2\int_{\mathbb{D}} \frac{\left(\det D\calS(z)\right)^{\frac{2}{3}} }{|\xi'-z|} (\det D\calS(z))^{\frac{1}{3}} |\omega(t,\calS(z)) - \omega(t',\calS(z)) | dz \\	
		\leq & 2\left(\int_{\mathbb{D}} |\xi'-z|^{-\frac 32} \det D\calS(z) dz \right)^{\frac{2}{3}} ||\omega(t,\cdot)-\omega(t',\cdot)||_{L^3(\Omega)}.  \lb{7.12}
\end{align}
(Note also that the first of these estimates and \eqref{7.9} below prove \eqref{1.6}.)
Since the last integral is bounded in $\xi'\in B'$ by Lemma \ref{L.3.1} and \eqref{6.0}, and $\omega$ is continuous as an $L^p(\Omega)$-valued function of $t\in[0,\infty)$ for any $p\in[1,\infty)$ due to boundedness of $\omega$, local boundedness of $u$, and \eqref{7.8}, these two estimates show that $R$ is continuous at $(t,\xi)$.
% (and hence on $[0,\infty) \times K$).

Boundedness of $R$ on $[0,\infty) \times K$ follows from the estimate
    \beq \lb{7.9}
    |R(t,\xi)| \leq C_{\Omega} ||\omega||_{L^\infty} (1-|\xi|)^{1-2\alpha_*}
    \eeq
for all $(t,\xi)\in[0,\infty)\times \bbD$, with $\alpha_*$ from \eqref{1.11} and some $\Omega$-dependent constant $C_{\Omega}$.  To obtain it, first note that $|z-z^*|\le 2|\xi-z^*|$ and \eqref{6.0} yield (with $\delta$ from \eqref{1.11})	
		\begin{align*}
 \int_{\Omega \backslash B(\xi,\delta)} \frac{|z-z^*|}{|\xi-z||\xi-z^*|} \det D\calS(z) dz 
%		&\leq 2||\omega_0||_{L^\infty} \int_{\Omega \backslash B(\xi,\delta)} \frac{1}{|\xi-z|} \det D\calS(z) dz \\
		\leq \frac 2\delta  \int_{\Omega \backslash B(\xi,\delta)}  \det D\calS(z) dz 
		\leq  \frac {2|\Omega|}\delta.
		\end{align*}
Then use  Lemma \ref{L.3.1}, and \eqref{2.1} with $H:=B(\xi,\delta)$, $I:=(\arg(\xi)-2\delta,\arg(\xi)+2\delta)$,  $f(z):=\frac{1}{|\xi-z|}$, and $\beta:=\sum_{\tht_j \in I} \alpha_j^+ \delta_{\theta_j}$ to get (with  $\eps:=\frac{1-|\xi|}2$ and $\tilde{\xi} = \frac{\xi}{|\xi|}$)
\begin{align*}
 \int_{B(\xi,\delta)} & \frac{|z-z^*|}{|\xi-z||\xi-z^*|}  \det D\calS(z) dz 
		 \leq  C' \int_{ B(\xi,\delta)} \frac{|\tilde{\xi}-z|^{-2\alpha_*}}{|\xi-z|}  dz \\
		&\leq C' \left(  \int_{B(\xi,\eps)}  \frac{\eps^{-2\alpha_*}}{|\xi-z|} dz  
		+ \int_{B(\tilde{\xi},\eps)}\frac{|\tilde{\xi}-z|^{-2\alpha_*}}{\epsilon}  dz
		+ 9 \int_{ B(\xi,\delta) \backslash (B(\xi,\eps) \cup B(\tilde{\xi},\eps) )} {|\xi-z|^{-1-2\alpha_*}}  dz   \right) \\
		&\leq C''(1-|\xi|)^{1-2\alpha_*}
\end{align*}		
with some $\Omega$-dependent constant $C',C''$ because $\sum_{\tht_j \in I} \alpha^+_j \le \alpha_*<1$ by \eqref{1.11}.   The last two estimates now imply \eqref{7.9}.  
%Hence it remains to show that $I \in C([0,T] \times K)$.

%    where we use inequality \eqref{6.4}, $|\xi-z|, |\xi'-z|, |\xi^*-z|, |\xi'^*-z| \geq \frac{d}{4}$ for $z \in \bbD \backslash B$ and $\int_{\bbD \backslash} \det D\calS(z) dz \leq |\Omega|$.
%    
%    For the second part we have $\omega \in L^q([0,T],\Omega)$ for all $1 \leq q < \infty$, for the integration, it is bounded uniformly for $\xi' \in B'$ since $\det D\calS(z)$ is bounded in $B$,  $\frac{1}{|\xi-z|^{1.5}}$ is bounded by $\frac{8}{d^{1.5}}$ on $\mathbb{D} \backslash B$ for all $\xi' \in B'$, and both are integrable on $\bbD$ (where the integration of $\frac{1}{|z-\xi'|^{1.5}}$ could be bounded by $8\pi$ for all $\xi' \in \bbD$). Combining all we have the $|R(t',\xi')-R(t,\xi)| \rightarrow 0$ as $(t',\xi') \rightarrow (t,\xi)$, hence $R \in C([0,T] \times K)$ and so is $u$.
    
Let us now turn to $Q$.
Fix any $K$ as above, then fix any $(t,\xi) \in [0,\infty) \times K$ and let $d,B,B'$ be as above (without loss assume that $d\le \frac 14$). Then for any $(t',\xi') \in [0,\infty) \times B'$ we have from  \eqref{2.3},
\begin{align*}
| Q(t,\xi) - 	Q(t,\xi')|	
%&\leq \int_{\mathbb{D}} \left(\frac{|\xi-\xi'|}{|\xi-z||\xi'-z|}    + \frac{|\xi-\xi'|}{|\xi-z^*||\xi'-z^*|} \right) \det D\calS(z) ||\omega_0||_{L^\infty} dz \\
\leq 
%& C ||\omega||_{L^\infty} \int_{B}\left(\frac{|\xi-\xi'|}{|\xi-z||\xi'-z|}    + \frac{|\xi-\xi'|}{|\xi-z^*||\xi'-z^*|} \right)  dz \\
		  ||\omega||_{L^\infty} \int_\mathbb{D} \left(\frac{|\xi-\xi'|}{|\xi-z||\xi'-z|}    + \frac{|\xi^*-\xi'^*|}{|\xi^*-z||\xi'^*-z|} \right) |R(t,z)| dz,
\end{align*}
where the second fraction is just $\frac 1{|\xi^*-z|}$ when $\xi'=0$ and $\frac 1{|\xi'^*-z|}$ when $\xi=0$.  Using \eqref{2.3}, splitting the integration to $z\in {B}$ and $z\in \mathbb{D} \backslash B$, and applying \eqref{7.9} and \eqref{6.4} yields
\[
| Q(t,\xi) - 	Q(t,\xi')|	 \le C' ||\omega||_{L^\infty} |\xi'-\xi| \left( d^{1-2\alpha_*} \left(1+\ln_+\frac{1}{|\xi-\xi'|}  \right) +d^{-2}\right)
\]
for some $\Omega$-dependent constant $C'$.  Next, we have
 \begin{align*}
| Q(t,\xi') - 	Q(t',\xi')| \le & ||\omega||_{L^\infty} \int_{\bbD} \frac{|\xi'-\xi'^*|}{|\xi'-z| \, |\xi'^*-z|}  | R(t,z) - R(t',z) |  dz \\
& + \int_{\bbD} \frac{|\xi'-\xi'^*|}{|\xi'-z| \, |\xi'^*-z|}  |R(t',z)| \, |\omega(t,\calS(z))- \omega(t',\calS(z)) | dz.
  \end{align*}
  Splitting the first integration into $z\in{B'}$ and $z\in \mathbb{D} \backslash B'$,  and then using $|\xi'-\xi'^*|\le 2  |\xi'^*-z|$,  \eqref{7.12}, and \eqref{7.9} shows that the first integral is bounded above by
    \[
    C_d  ||\omega(t,\cdot)-\omega(t',\cdot)||_{L^3(\Omega)} + \frac 4d \int_{\bbD} | R(t,z) - R(t',z) |  \,dz
\]
 for some $(\Omega,d)$-dependent constant $C_d$.  This converges to 0 as $t'\to t$ by  continuity of $\omega:[0,\infty)\to L^3(\Omega)$, together with \eqref{7.12} and integrability of the right-hand side of \eqref{7.9}.
 
Using $|\xi'-\xi'^*|\le 2  |\xi'^*-z|$, \eqref{7.9}, and Lemma \ref{L.3.1}, the second integral is bounded by
    \begin{align*}
C'  \left[ \int_{\bbD}  \left(\frac{(1-|z|)^{1-2\alpha_*}}{|\xi'-z|\det D\calS(z)^\frac{1}{p}} \right)^q  dz\right]^{\frac{1}{q}}   &\left(\int_{\bbD} \det D\calS(z)|\omega(t,\calS(z)) - \omega(t',\calS(z)) |^p dz     \right)^{\frac{1}{p}} \\
& \le C_d ||\omega(t,\cdot)-\omega(t',\cdot)||_{L^p(\Omega)}
  \end{align*}
  for some $\Omega$-dependent $C'$ and $(d,\Omega)$-dependent  $C_d$, provided $p\in(2,\infty)$ is large enough so that with $q:=\frac p{p-1}$ we have $(1-2\alpha_*-\frac1p {\sum_{j} \alpha_j^+})q>-1$.  
  The above estimates thus together show that $Q$ is continuous at $(t,\xi)$.

We can also use  \eqref{2.3}, $|\xi-\xi^*|\le 2  |\xi^*-z|$, and \eqref{7.9} to get
\beq\lb{boundQ}
|Q(t,\xi)| \le  2C_\Omega  ||\omega||_{L^\infty}^2 \int_{\bbD}  \frac{(1-|z|)^{1-2\alpha_*}}{|\xi-z|}   dz
\eeq
for all $(t,\xi)\in[0,\infty)\times\bbD$, showing boundedness of $Q$ on $[0,\infty)\times K$ for each compact $K\subseteq\bbD$.
  
Hence it remains to show  $\partial_t  \Psi(t,x)=- \frac 1{2\pi} Q(t,\calT(x))$ pointwise, which will follow from
	\beq\lb{8.1}
	-\frac 1{2\pi}\int_{t_0}^{t_1} Q(t,\calT(x_0))dt = \Psi(t_1,x_0) - \Psi(t_0,x_0)
	\eeq
	for all $0\le t_0<t_1$ and $x_0\in\Omega$ because $Q$ is continuous.  So fix any such $(t_0,t_1,x_0)$.

Let
\[
\phi(x) :=  - \frac 1{2\pi} \ln \frac{|\calT(x_0)-\calT(x)|}{|\calT(x_0)-\calT(x)^*||\calT(x)|} 
= - \frac 1{2\pi}  \ln \frac{|\calT(x)-\calT(x_0)|}{|\calT(x)-\calT(x_0)^*||\calT(x_0)|}
\]
(so $\Psi(t_j,x_0)= \int_\Omega \phi(x) \omega(t_j,x)dx$ for $j=0,1$) and
\[
\psi(x) := \nabla \phi(x) = - \frac 1{2\pi} D\calT(x)^T \left( \frac{\calT(x)-\calT(x_0)}{|\calT(x)-\calT(x_0)|^2} - \frac{\calT(x)-\calT(x_0)^*}{|\calT(x)-\calT(x_0)^*|^2} \right)
\]
for each $x\in\Omega$ (recall \eqref{9.1}).  Also, for each $r\in(0,\frac{t_1-t_0}2)$ let $g_r\in C_c^{\infty}([0,\infty))$ be such that
\[
 \chi_{[t_0+r,t_1-r]} \le g_r \le  \chi_{(t_0,t_1)}
\]
and $g_r$ is non-increasing on $[0,t_1]$ and non-decreasing on $[t_1,\infty)$;
and for each $h\in(0,1]$ let $f_h\in C^\infty([0,\infty))$ be such that
\begin{enumerate}
	\item $f_h(x) = 0$ for  $x \in [0,\frac{h}{3}]$,
	\item $f_h(x) = x$ for $x \in [h,\frac{1}{h}]$,
	\item $f_h(x) = \frac{1}{h}+h$ for $x \in [\frac{1}{h}+h,\infty)$,
	\item $0\le f_h'(x) \leq 2$ for $x \in [0,\infty)$.
\end{enumerate}

Now for any $h,r\in(0,\min\{1,\frac{t_1-t_0}2\})$ and $(t,x)\in [0,\infty) \times \Omega$ let 
\[
\varphi_{r,h}(t,x) := g_r(t) f_h(\phi(x)).
\]
Then clearly $\varphi_{r,h}\in C^{\infty}_c([0,\infty) \times \Omega)$ and $\varphi_{r,h}(0,\cdot)\equiv 0$, so plugging it into \eqref{1.222} yields 
\[
\int_{0}^\infty  \int_{\Omega} \omega(t,x)  g_r(t) f_h'(\phi(x))\, u(t,x)\cdot \psi(x) dxdt  + \int_{0}^\infty \int_{\Omega} \omega(t,x) g_r'(t) f_h(\phi(x))  dxdt =0.
\]
%when $t_0>0$, and  with $\int_\Omega \omega(0,x) f_h(\phi(x)) dx$ added to the left-hand side when $t_0=0$.  
Since $\omega(t,x) g_r(t) f_h'(\phi(x)) \psi(x)$ is a bounded function and $u\in L^{\infty}((0,\infty); L^2(\Omega))$, we can use the dominated convergence theorem to pass to the limit $r\to 0$ and obtain
\[
\int_{t_0}^{t_1}  \int_{\Omega} \omega(t,x)  f_h'(\phi(x))\, u(t,x)\cdot \psi(x) dxdt  +\int_{\Omega} \omega(t_0,x)  f_h(\phi(x))  dx - \int_{\Omega} \omega(t_1,x)  f_h(\phi(x)) dx =0,
\]
where in the second integral above we used that  $\omega$ is continuous as an $L^1(\Omega)$-valued function of $t\in[0,\infty)$.  If we can show that $u\cdot\psi\in L^{\infty}((0,\infty); L^1(\Omega))$, then taking $h\to 0$ will yield
\[
\int_{t_0}^{t_1}  \int_{\Omega} \psi(x)^T u(t,x) \, \omega(t,x)  dxdt  = \int_{\Omega} \phi(x) \omega(t_1,x)   dx - \int_{\Omega} \phi(x) \omega(t_0,x)  dx
\]
via the dominated convergence theorem.  But this is precisely \eqref{8.1} due to \eqref{2.2} and \eqref{7.18}.

If $B:=B(x_0,\frac 12 {\rm dist}(x_0,\partial \Omega))$, then  $u\cdot\psi\in L^{\infty}((0,\infty); L^1(B))$ because $u$ is bounded on $[0,\infty)\times B$ by \eqref{7.9}.  From \eqref{9.1} we see that there is $C_{x_0}$ such that 
\[
|\psi(x)|\le C_{x_0} \|D\calT(x)\|\le 2 C_{x_0} |\det D\calT(x)|^{\frac 12}
\]
for all $x\in \Omega\setminus B$, so $\psi\in L^2(\Omega)$ by $\int_\Omega \det D\calT(x)dx=|\bbD|$.  So $u\cdot\psi\in L^{\infty}((0,\infty); L^1(\Omega\setminus B))$, which indeed yields $u\cdot\psi\in L^{\infty}((0,\infty); L^1(\Omega))$ and thus finishes the proof.

\end{document}